\documentclass[a4paper]{article}

\usepackage{graphicx}
\usepackage[utf8]{inputenc}
\usepackage{amsmath, amsthm, amssymb, amsfonts,mathtools}

\usepackage{fourier}
\usepackage[french,english]{babel}
\usepackage{microtype}
\usepackage[backend=bibtex,style=alphabetic]{biblatex}

\newtheorem{theorem}{Theorem}[section]
\newtheorem{lemma}[theorem]{Lemma}
\newtheorem{proposition}[theorem]{Proposition}
\newtheorem*{proposition*}{Proposition}
\newtheorem*{theorem*}{Theorem}
\newtheorem{corollary}[theorem]{Corollary}

\newtheorem{conjecture}[theorem]{Conjecture}

\usepackage{pythonhighlight}

\title{Redundancy of  triangle groups in spherical CR representations}

\author{Raphaël V.~{\sc Alexandre}\footnote{Institut de Math\'ematiques de Jussieu-Paris Rive Gauche, Sorbonne Université, 4 place Jussieu, 75252 Paris Cédex, France.
ACG,
OURAGAN (IMJ-PRG, INRIA Paris, Sorbonne Université, Université de Paris, CNRS).
Email address: {\tt raphael.alexandre@math.cnrs.fr}.}}

\AtBeginDocument{\maketitle}

\newcommand\textthmrest{
Consider the following table.
Each row consists of manifolds having a boundary  unipotent representation with   the group  $\Delta_2(3,3,n;\theta_\infty)$ denoted in the first column for image (up to conjugation in $\PU(2,1)$ and complex conjugation). 

The manifolds in the column “Uniformization” are (spherical CR) uniformized by their representation and the others in the column “Redundant” are not.
}
\newcommand\experimentaltable{
\begin{center}
\begin{tabular}{|c|c|c|}
\hline
Group & Uniformization & Redundant
\\
\hline
$\Delta_2(3,3,4;\theta_\infty)$
& m004
& m022, m029, m034, m053, m081, m117
\\ 
$\Delta_2(3,3,5;\theta_\infty)$
& m009
& m015, m035, m142, m146
\\
$\Delta_2(3,3,6;\theta_\infty)$
& m023
& 
\\
$\Delta_2(3,3,7;\theta_\infty)$
& m039
& m032, m045
\\
$\Delta_2(3,3,8;\theta_\infty)$
& s000
& m053
\\
$\Delta_2(3,3,\infty;\theta_\infty)$
& m129 & m203
\\
\hline
\end{tabular}
\end{center}
}

\newcommand\Z{\mathbf{Z}}
\newcommand\R{\mathbf{R}}
\newcommand\C{\mathbf{C}}

\DeclareMathOperator{\SU}{SU}
\DeclareMathOperator{\PGL}{PGL}
\DeclareMathOperator{\PU}{PU}
\DeclareMathOperator{\rU}{U}
\DeclareMathOperator{\tr}{tr}
\renewcommand{\Re}{{\rm Re}}

\newcommand{\ii}{\boldsymbol{i}}

\bibliography{ref.bib}

\begin{document}

\begin{abstract}
Falbel, Koseleff and Rouillier computed a large number of boundary unipotent CR representations of fundamental groups of non compact three-manifolds. Those representations 
are not always discrete. 
By experimentally computing their limit set, one can determine that those with fractal limit sets are discrete. 
Many of those discrete representations can be related to $(3,3,n)$ complex hyperbolic triangle groups. 
By exact computations, we verify the existence of those triangle representations, which have boundary unipotent holonomy.
We also show that many representations are redundant: for $n$ fixed, all the $(3,3,n)$ representations encountered are conjugate and only one among them is uniformizable.
\end{abstract}
\begin{center}
\end{center}

\tableofcontents

\section{Introduction}

Falbel, Koseleff and Rouillier~\cite{FKR} explicitly computed a large number
of  fundamental group representations of knot and link complements in $\PGL(3,\C)$ and in particular in $\PU(2,1)$ (it is “a great portion” among all, according to~\cite{FKR}). With those numerical methods we have \emph{all} the  boundary unipotent  representations in $\PU(2,1)$  for the complements described by  four  or less tetrahedra. They were made available in~\cite{Ptolemy} with the collaboration of~\cite{Garou1,Garou2}. 

A   boundary unipotent representation of $M$ is a representation of the fundamental group $\pi_1(M)$ such that each subgroup corresponding to each cusp is sent to an abelian subgroup of $\PU(2,1)$ generated by one or two  unipotent transformations.

With those representations, one can raise some delicate questions. \emph{Which  are discrete? Which  have an image that is a  complex hyperbolic triangle group up to finite index?}

\vskip10pt

The group $\PU(2,1)$ is the holomorphic isometry group of the \emph{complex hyperbolic plane} $\mathbf H_\C^2$.
Complex hyperbolic triangle groups will be defined and discussed in section~\ref{sec-2-triangle}.  For the moment, we see them as the representations of the
\emph{abstract triangle groups}
\begin{equation}
\Lambda(p,q,r) = \left\langle a,b,c \; \middle\vert \;
\begin{matrix} 
&a^2=b^2=c^2=e, \\
&(ab)^p=(bc)^q=(ca)^r=e
\end{matrix}\right\rangle, \; 2\leq p\leq q \leq r \leq \infty,
\end{equation}
into $\PU(2,1)$ by complex reflections, with $\frac1p+\frac1q+\frac1r<1$.
Such a representation will be called a 
\emph{complex hyperbolic triangle group}, denoted by $\Delta(p,q,r;\theta)$, and the images of $a$, $b$ and $c$ are  denoted by $I_1$, $I_2$ and $I_3$.  There is a parameter $\theta$ describing the space of the representations, it is the \emph{angular} (or \emph{Cartan})  \emph{invariant}.

There is an important   example of a three-manifold having a representation of its fundamental group with a triangular image.  Falbel~\cite{Falbel} constructed the boundary unipotent  representations of the fundamental group of the figure-eight knot complement. 
There are three such representations and they are discrete. Among them,  two  are also (index two) complex hyperbolic triangle groups by~\cite{DF}. To be more specific, those two representations can be identified with the normal subgroup of the even-length words  of a complex hyperbolic triangle group $\Delta(3,3,4;\theta)$. We will denote  index two normal subgroup by $\Delta_2(p,q,r;\theta)\subset\Delta(p,q,r;\theta)$.

Complex hyperbolic triangle groups furnish a tiny subset of infinite covolume discrete subgroups of $\PU(2,1)$. They are Coxeter groups and are usually recognized as surface groups. But we expect it to be  a very rich class of representations even in the world of three-manifolds.

\vskip10pt

In our study, we take  a boundary unipotent representation $\rho$ and ask wether it is discrete and if its image is in a complex hyperbolic triangle group.  

To study the discreteness of a representation we chose to numerically approximate   its limit set. This set is an attractor for the iteration dynamic and a simple argument allows us to support the discreteness: if the limit set is fractal (that is to say, is neither dense in $S^3$ nor a smooth circle) then the representation is discrete.

\vskip10pt
Stored in SnapPy~\cite{Ptolemy} by a joint work of~\cite{FKR} with~\cite{Garou1,Garou2}, we have at our disposition $1653$ boundary  unipotent representations.
By inspection of all the approximated limit sets, we found 
$35$ 
pairs of representations (paired by complex conjugation in $\PU(2,1)$) that have  a fractal limit set.\footnote{We also observed that every representation from the census that gave a fractal had each boundary component (a copy of $\Z^2$) sent to a \emph{rank one} parabolic subgroup.}
They concern $20$ manifolds of the census. In figures~\ref{fig-all-1} and~\ref{fig-all-2} (in the appendix)  we exposed all those fractal limit sets.

Among those 
$35$
pairs of representations,  
$21$  are in fact derived from complex hyperbolic triangle groups $(3,q,r)$. 
They come from $19$ manifolds.
And among those $21$ triangle representations, $19$ are derived from  $(3,3,n)$ complex hyperbolic triangles. 

On the remaining $14$ pairs, we show that 
$3$
 are not surjective morphisms into an index two $(3,3,n)$-triangle group (see proposition~\ref{prop-counterex}) and that an additional one is a triangular representation, but different from the others (see proposition~\ref{prop-m023}): it is a \emph{Lagrangian} triangle group (as defined in~\cite{Will07}).

\vskip10pt

We will (mathematically) prove that those $19$ triangle representations
do in fact come from index two $(3,3,n)$ complex hyperbolic triangle groups.
All those representations are discrete.

 We show that the conjugacy class can be chosen so that $I_3I_2I_1I_2$ generates the boundary holonomy of the fundamental group. This phenomenon is  not true  in general (see proposition~\ref{prop-m023}).
 For a fixed abstract triangle group, the transformation $I_3I_2I_1I_2$ is unipotent for a unique value of the angular invariant $\theta$ that we will denote $\theta_\infty$.

For each of those triangle representations, we find that the  image group is the even-length words subgroup  $\Delta_2(p,q,r;\theta_\infty)\subset\Delta(p,q,r;\theta_\infty)$. 
As it is suggested by its importance,
we consider $\Lambda_2(p,q,r)$ (the abstract version of $\Delta_2(p,q,r;\theta_\infty)$) defined by:
\begin{equation}
\Lambda_2(p,q,r) = \left\langle x,y \; \middle\vert \; x^p=y^q=(xy)^r=e\right\rangle.
\end{equation}

The inclusion $\Delta_2(p,q,r;\theta)\subset\Delta(p,q,r;\theta)$ corresponds to an inclusion morphism $\Lambda_2(p,q,r)\to\Lambda(p,q,r)$ defined by  $x\mapsto ab$ and $y\mapsto bc$.

\paragraph{Redundancy}
We call \emph{redundancy} the phenomenon of having several manifold fundamental groups represented in a same $\Delta_2(p,q,r;\theta_\infty)$, or more generally in a same $\Lambda_2(p,q,r)$. To assert the equality of the images, we allow conjugacy in $\PU(2,1)$ and complex conjugation.

Deraux~\cite{Deraux}  first showed that m009 and m015 are two manifolds with representations  in a same $\Delta_2(3,3,5;\theta_\infty)$. But only the representation of m009 is a \emph{spherical CR uniformization}.

The \emph{spherical CR structure} is the pair $(\PU(2,1),\partial \mathbf H_\C^2)$.
There is a delicate relationship between a representation  of a manifold $M$ (of its fundamental group) in $\PU(2,1)$ and a \emph{uniformizable} spherical CR representation of $M$. In the latter case, the image group $\Gamma$ completely determines $M$: if $U\subset\partial\mathbf{H}^2_\C$ is its discontinuity domain then $U/\Gamma$ is diffeomorphic to $M$ (that is our definition of being uniformizable, as in~\cite{Deraux}). 

It is very hard in general  to determine if a representation in $\PU(2,1)$ is the holonomy of a spherical CR uniformization.   In our study, we  relied on~\cite{Acosta} to identify uniformizations.

\begin{theorem*}[\ref{thm-rest}]
\textthmrest
\end{theorem*}

\experimentaltable

\paragraph{Notes}
We use the nomenclature of SnapPy for the three-manifolds “mxxx”.

Our proof relies on the  computation of a surjective morphism from $\pi_1(M)$ to $\Lambda_2(3,3,n)$. Indeed, we can observe that the data of a surjective morphism into a $\Lambda_2(p,q,r)$ always furnishes a surjective representation into the corresponding $\Delta_2(p,q,r;\theta_\infty)$ since $\Lambda_2(p,q,r)\to\Delta_2(p,q,r;\theta_\infty)$ is always surjective. It will remain to check that the peripheral holonomy  is indeed unipotent.

We cannot \emph{mathematically} prove that the representations appearing in the theorem are the only ones that can be found in the census~\cite{Ptolemy}. There might be other representations in the census that are triangular by a $\Delta_2(3,3,n;\theta_\infty)$ triangle group.
But we are strongly confident that we found all of them by the approximation of the limit sets.

A consequence of our work is that the property of having a surjective morphism onto a $\Lambda_2(p,q,r)$ is not a clearly understood property: some manifolds possess several (\emph{e.g.} m053) and redundancy frequently  occurs. 

More recently, Will~\cite{WillHdr} produced a list of more manifolds admitting a representation into $\Lambda_2(3,3,\infty)\cong \mathbf Z_3*\mathbf Z_3$ and even $\Lambda_2(p,q,\infty)\cong \mathbf Z_p* \mathbf Z_q$. Most manifolds appearing in his list are not contained in the census~\cite{Ptolemy}.

\paragraph{Deformations}
 Let $M$ have a surjective morphism $\rho\colon\pi_1(M)\to\Lambda_2(p,q,r)$. Then this (abstract) representation gives various  representations $\pi_1(M)\to\Lambda_2(p,q,r)\to \PU(2,1)$. 
The space of all those representations up to conjugation in $\PU(2,1)$ has real dimension $2$. The principal invariant to distinguish two representations is the trace of the image of $xy^{-1}\in\Lambda_2(p,q,r)$. (See the proposition~\ref{prop-Lawton} and the following discussion for details.)

The manifolds considered in theorem~\ref{thm-rest} that  originally had a few discrete unipotent representations now automatically get a $2$-dimensional subset in their character variety. In general, the character variety has dimension at least two. When it is minimal, as with the figure-eight knot complement, then this construction gives an open component (compare with~\cite{FalbelAl}). 

Numerical approximations of  deformations of  some triangle groups are available on the webpage of the author~\cite{rph}.

\vskip10pt

In most cases, a surjective morphism $\pi_1(M)\to\Lambda_2(3,3,n)$ is represented in $\PU(2,1)$ by $\Delta_2(3,3,n;\theta_\infty)$ and gives a boundary  unipotent representation of $M$.
But this phenomenon is not always true and reveals that complex hyperbolic triangles are not the only ones to be considered.

\begin{proposition*}[\ref{prop-m023}]
There exists a surjective morphism 
\begin{equation}
\xi\colon \pi_1({\rm m023}) \to \Lambda_2(3,3,4),
\end{equation}
and a composition $\pi_1({\rm m023})\to\Lambda_2(3,3,4)\to\PU(2,1)$ having boundary unipotent holonomy.
The boundary unipotent representation in $\PU(2,1)$ is not induced by any representation $\Lambda_2(3,3,4)\to\Delta_2(3,3,4;\theta)$.
\end{proposition*}

In this case,
 the boundary holonomy is controlled by the element $[x,y]\in \Lambda_2(3,3,4)$. 
(In theorem~\ref{thm-rest} the boundary holonomy is always controlled by  $xy^{-1}\in \Lambda_2(3,3,n)$.)
But for any of  representation $\Lambda_2(3,3,4)\to\Delta_2(3,3,4;\theta)$, the word $[x,y]$ is never sent to a unipotent element of $\PU(2,1)$. This will justify the proposition.

\paragraph{Outline of the paper}
In section~\ref{sec-2}, we succinctly expose a few elements of complex hyperbolic geometry that are needed in the paper. 
In section~\ref{sec-3}, we explain how the numerical experiments were driven. We also propose visual clues in order to recognize the limit sets of the various triangle representations $\Delta_2(3,3,n;\theta_\infty)$.
In section~\ref{sec-4} we state and show the main result, theorem~\ref{thm-rest}. For this, we only employ formal computations, so the result is certain. 
In section~\ref{sec-5} we give other examples. Three examples for which no surjective morphism to $\Lambda_2(3,3,n)$ does exist. We also show the proposition~\ref{prop-m023}, stated earlier.

\paragraph{Methodology}
In this paper we deal with two kinds of evidence. What we call \emph{experimental} arguments are the one arising from the numerical approximations of the limit sets. Those explain how we selected the  candidates for triangle representations.

But the results  and in particular theorem~\ref{thm-rest} and proposition~\ref{prop-m023}  are entirely based on pure mathematical arguments. In fact, they only depend on the description of the fundamental groups of the cited manifolds. Proofs can be checked by hand.

The identification of a representation arising from theorem~\ref{thm-rest} with  a corresponding one in the census~\cite{Ptolemy}  is made by an exact computation with the formal tools provided by Python and SageMath. 
The census has its representations stored as matrices with entrees that are polynomial or rational with rational coefficients in an algebraic field. 
Therefore the representations can be manipulated with exactness. 
It is only when we numerically approximate the limit sets that we numerically approximate the representations.

We will respect the following rule.
\emph{The stated theorems, propositions and lemmas are always mathematical.} That is to say, proven by hand or by formal (exact) computations.

\paragraph{Note}This paper is part of the author's thesis, in progress under the supervision of Elisha Falbel.
\paragraph{Acknowledgments}The author enjoyed many and very fruitful conversations. 
First of all, of course, I am very thankful to Elisha Falbel. 
Since the early stages in making the experimental tools, Fabrice Rouillier and Antonin Guilloux have been of precious help for the improvement  of my code. Across countries, Mathias Görner has been essential to me in order to correctly use the tools provided by SnapPy. I would  like to thank Pierre Will for all his comments and the many hours of discussions we had. I have been pleased to exchange with Miguel Acosta about his theorem and about experimental aspects.   Finally, I would like to thank the reviewers who gave me many comments significantly improving the text.

\section{Elements of complex hyperbolic geometry}\label{sec-2}

In this first section, we expose the main tools and notions in use. One can compare with~\cite{Will}, \cite{Goldman}, \cite{Pratoussevitch} and~\cite{ChenGreenberg}.

We consider the space $\C^{2,1}$: it is $\C^3$ equipped with the Hermitian product of signature $(2,1)$
\begin{equation}
\langle z,w\rangle = z_1\overline{w_1} + z_2\overline{w_2} - z_3\overline{w_3}.
\end{equation}
The subspace of the vectors verifying $\langle z,z\rangle <0$ can be projectivised in $\C\mathbf P^2$ and is identified to the \emph{complex hyperbolic plane}, $\mathbf H_\C^2$. In the affine chart $z_3=1$, one can identify $\mathbf H_\C^2$ with the set of vectors verifying $|z_1|^2 + |z_2|^2 < 1$.  Its boundary in the complex projective plane is a differentiable sphere $S^3$ and is given by $|z_1|^2 + |z_2|^2 = 1$. Those points are in correspondance with the non-zero vector lines in $\C^{2,1}$ verifying $\langle z,z\rangle = 0$.

The unitary group of $\C^{2,1}$ is $\rU(2,1)$ and its projectivised version is $\PU(2,1)$. The  geometrical structure  $(\PU(2,1),S^3)$ is called the \emph{(holomorphic) spherical CR structure}.

Let $A \in \SU(2,1)$. 
If $A$ has a fixed point in $\mathbf H_\C^2$ then $A$ is \emph{elliptic}. 
If $\inf \{d(x,A(x))\}>0$ with $d$ the  associated distance function of $\mathbf H_\C^2$
then $A$ is \emph{loxodromic} (or \emph{hyperbolic}). Otherwise, $A$ is \emph{parabolic}. One can determine the type of $A$ by looking at its trace. We follow Goldman~\cite{Goldman} and let
\begin{equation}
f(\tau) = |\tau|^4 - 8 \Re(\tau^3) +18|\tau|^2 -27.
\end{equation}
If $f(\tr A)> 0$ then $A$ is loxodromic, if $f(\tr A)<0$ then $A$ is elliptic (in fact regular elliptic: all its eigenvalues are different).

When $f(\tr A)=0$ there are three cases: if $\tr(A)^3=27$ then $A$ is (parabolic) unipotent (all its eigenvalues are $1$), otherwise it is either elliptic (and therefore a reflection with respect to a point or a complex geodesic) or ellipto-parabolic (a screw transformation along a complex geodesic). Note that when $\tau$ is real:
\begin{equation}
f(\tau) = (\tau+1)(\tau-3)^3,
\end{equation}
and (under the hypothesis that $\tr(A)$ is real) $A$ is  therefore  regular elliptic if $\tr(A)\in ]-1,3[$, is loxodromic if $\tr(A)\not\in [-1,3]$ and is  unipotent if $\tr(A)=3$.
Note that unipotent transformations are a special kind of parabolic elements.

\vskip10pt
Let $M$ be a smooth manifold and $\pi_1(M)$ its fundamental group. A representation $\rho\colon \pi_1(M) \to \PU(2,1)$ is the holonomy of a (CR) \emph{uniformization} of $M$ if, with  $U\subset \partial \mathbf H_\C^2$ the domain of discontinuity of $\rho(\pi_1(M))$, the quotient  $U/\rho(\pi_1(M))$ (that is generally only an orbifold) is a manifold diffeomorphic to $M$. When $\rho$ is discrete, $U = \partial \mathbf H_\C^2 - L(\rho(\pi_1(M)))$, where $L(\rho(\pi_1(M)))$ is the \emph{limit set} of $\rho(\pi_1(M))$. The next section will describe this set. A manifold admitting such a representation is said to be (CR) \emph{uniformizable}. Those manifolds are central in the study of  spherical CR structures and are determined by the algebraic data of  $\rho$. 
In general, even when $U/\rho(\pi_1(M))$ is a smooth manifold, it is too hard to identify it. It remains unknown which three-manifolds are CR uniformizable.

\subsection{Limit sets}

Let $\Gamma\subset \PU(2,1)$ be a subgroup. Its \emph{limit set} $L(\Gamma)$ is given by:
\begin{equation}
L(\Gamma) = \overline{\Gamma\cdot p}\cap \partial \mathbf H_\C^2,
\end{equation}
where $p\in \mathbf H_\C^2$ is any point ($L(\Gamma)$ is independent of this choice) and  where the overline denotes the topological closure in $\overline{\mathbf H_\C^2}\subset \C\mathbf P^2$.

\begin{lemma}
The main properties of this set are the following. (Compare with~\cite{ChenGreenberg}.)
\begin{enumerate}
\item The limit set $L(\Gamma)$ is compact and $\Gamma$-invariant.
\item If $A\subset \partial \mathbf H_\C^2$ is compact, $\Gamma$-invariant and is  constituted of at least two points, then $L(\Gamma)\subset A$.
\item If $L(\Gamma)=\emptyset$ then $\Gamma$ fixes a point in $\mathbf H_\C^2$.
\item If $L(\Gamma)$ has at most two points, then $L(\Gamma)$ (or $\Gamma$) is said to be \emph{elementary}, otherwise it has an infinite number of points and is perfect (each point is an accumulation point).
\end{enumerate}
\end{lemma}

An important result is the following.
\begin{proposition*}[\cite{ChenGreenberg}]
If $\Gamma$ is not discrete then $L(\Gamma)$ is either elementary, or equal to $\partial \mathbf H_\C^2$, or equal to the boundary of a totally geodesic proper subspace, that is to say a smooth circle.
\end{proposition*}

By consequence, if $L(\Gamma)$ is a fractal (none of the above cases), then $\Gamma$ is discrete. It is a powerful experimental way to investigate if $\Gamma$ is discrete. Note that there is no general mathematical procedure to prove with certainty that a subgroup is discrete or not.

The  self-similarity property of limit sets can be justified by the following lemma.
\begin{lemma}
Let $\Gamma$ be a discrete subgroup of $L(\Gamma)$ and suppose that $L(\Gamma)$ is not elementary. Let $a\in L(\Gamma)$ be any point and $V$ be any open neighborhood of $a$. Then there exists $\gamma_1,\dots,\gamma_n\in\Gamma$ such that
\begin{equation}
L(\Gamma)= \bigcup_{i=1}^n \gamma_i\cdot \left(V \cap L(\Gamma)\right).
\end{equation}
\end{lemma}

\begin{proof}
Let $W = \partial \mathbf H_\C^2 - \bigcup \Gamma\cdot V$. It is compact and $\Gamma$-invariant. By construction, $W$ cannot have more than one point. If $W=\{b\}$ then $\Gamma$ fixes $b$ and  $L(\Gamma)$ must be elementary since it is discrete.  Therefore $W=\emptyset$ and it follows that $L(\Gamma)\subset \bigcup \Gamma\cdot V$. By compactness of $L(\Gamma)$, only a finite number of $\gamma_i\in \Gamma$ are necessary.
\end{proof}

\subsection{Complex hyperbolic triangle groups}\label{sec-2-triangle}

We will now describe more precisely the \emph{complex hyperbolic triangle groups}
\begin{equation}
\Delta= \left\langle I_1,I_2,I_3 \; \middle\vert\;
\begin{matrix} 
&I_1^2=I_2^2=I_3^2=e, \\
&(I_1I_2)^p=(I_2I_3)^q=(I_3I_1)^r=e
\end{matrix}\right\rangle \subset \PU(2,1),
\end{equation}
with $I_1$, $I_2$ and $I_3$ all three being complex reflections. The numbers $p,q,r$ are integers possibly infinite. If one of those is infinite, then the corresponding relation is omitted.

We will prove that $\Delta$ only depends on $p,q,r$ and an additional real parameter $\theta$, up to conjugation and complex conjugation. This parameter $\theta$ will be determined by the value of the trace of $I_3I_2I_1I_2$.
In particular,  for any triplet $(p,q,r)=(3,3,r)$, we obtain a unique complex hyperbolic triangle group, up to conjugation and complex conjugation, with $I_3I_2I_1I_2$ unipotent. This will justify the notation $\Delta(p,r,q;\theta)$ to identify a complex hyperbolic triangle group.

\vskip10pt
Recall from the introduction that we can see a complex hyperbolic triangle group as a representation in $\PU(2,1)$ by complex reflection of the \emph{abstract} triangle group
\begin{equation}
\Lambda(p,q,r) = \left\langle a,b,c \; \middle\vert \;
\begin{matrix} 
&a^2=b^2=c^2=e, \\
&(ab)^p=(bc)^q=(ca)^r=e
\end{matrix}\right\rangle.
\end{equation}

When complex hyperbolic triangle groups were exposed by Schwartz~\cite{Schwartz}, he asked the following: when is a complex hyperbolic triangle group  a discrete and injective representation of the corresponding abstract triangle group? He proposed the following important conjecture.

\begin{conjecture}[Schwartz]
A complex hyperbolic triangle group $\Delta(p,q,r;\theta)$ is a discrete and injective representation of $\Lambda(p,q,r)$ in $\PU(2,1)$, if and only if $I_3I_2I_1I_2$ and $I_1I_2I_3$ are both not elliptic. 
\end{conjecture}

Note that, in some rare cases, $\Delta(p,q,r;\theta)$ can remain discrete (but not injective) with $I_3I_2I_1I_2$ and $I_1I_2I_3$ elliptic of finite order.
For example, Thompson~\cite{Thompson} found a representation $\Delta(3,3,4;\theta_7)$ with $I_3I_2I_1I_2$ of order $7$ and a representation $\Delta(3,3,5;\theta_5)$ with $I_3I_2I_1I_2$ of order $5$, and  both representations are lattices of $\PU(2,1)$.

Schwartz~\cite{Schwartz3} has shown his conjecture when $\min(p,q,r)$ is large enough.
A first step toward the general case is a result due to Grossi~\cite{Grossi}. In particular for $(p,q,r)=(3,3,n)$, Grossi shows that if $I_3I_2I_1I_2$ is not elliptic then $I_1I_2I_3$ is not either. A proof of Schwartz' conjecture in the case of $(3,3,n)$ has been given by Parker, Wang and Xie~\cite{ParkerWangXie} and the case of $(3,3,\infty)$ has been studied by Parker and Will~\cite{ParkerWill}.

\begin{theorem}[\cite{ParkerWangXie},\cite{ParkerWill}]\label{thm-schwarz33n}
Let $4\leq n \leq \infty$. Let $\Delta(3,3,n;\theta)$ be a hyperbolic $(3,3,n)$ triangle group. Then $\Delta(3,3,n;\theta)$ is  discrete and faithful   if and only if $I_3I_2I_1I_2$ is not elliptic.
\end{theorem}

\vskip10pt
Now we proceed to a description of complex hyperbolic triangle groups.
Let $\Delta$ be a  complex hyperbolic triangle group, with $2 \leq p\leq q \leq r\leq \infty$ and $\frac \pi p +\frac \pi q + \frac \pi r <\pi$. Each reflection $I_k$ fixes a complex line (isomorphic to $\mathbf H_\C^1$) in $\mathbf H_\C^2\subset \C\mathbf P^2$.
Let $H_1$, $H_2$ and $H_3$ be the vector hyperplanes of $\C^3$ covering those complex lines.

Let $L_1$, $L_2$ and $L_3$ be the dual complex lines of those hyperplanes defined by $\langle H_k, L_k\rangle = 0$. The group $\Delta$ is fully described by them.
Since $H_k$ is a complex line of $\mathbf H_\C^2$, it is a negative-type complex plane of $\C^3$ ($\langle h,h\rangle<0$), and $L_k$ is a positive-type complex line of $\C^3$ ($\langle l,l\rangle>0$).

We only need to choose a basis vector for each $L_i$ in order to describe those lines. 
We say that the triangle group is \emph{non-degenerate} when 
such basis vectors form a basis of $\C^3$. We will make this assumption from now on.

Let $v_k$ be a basis vector of $L_k$, then
\begin{equation}
I_k(x) = -x + \frac{2\langle x,v_k\rangle}{\langle v_k,v_k\rangle} v_k
\end{equation}
is a complex reflection  and verifies $I_k(h_k) = -h_k$ for any $h_k\in H_k$. That is to say, in $\C\mathbf P^2$, $[I_k(h_k)] = [h_k]$. Therefore $I_k$ indeed defines the reflection fixing $H_k$.
Because $\langle v_k,v_k\rangle >0$, one can normalize $v_k$ so that $\langle v_k,v_k\rangle = 1$. 

The last free parameters describing the $v_k$'s are an angle $z_k\in S^1$ for each $v_k$. One can set $z_1$ and then modify $z_2$ and $z_3$ so that $\langle v_1,v_2\rangle$ and $\langle v_2,v_3\rangle$ are real and positive.
In general, $\langle v_1,v_3\rangle$ is not real and this lack can
be measured by $\arg(\langle v_1,v_3\rangle)$. From an intrinsic  point of view, that is to say without choosing the $z_k$'s, the default for the vertices to be in a same real plane can be measured by
\begin{equation}
\theta = -\arg(\langle v_1,v_2\rangle \langle v_2,v_3\rangle\langle v_3,v_1\rangle).
\end{equation}
The value of $\theta$ is also known under the name of the \emph{angular invariant}.

Once $\langle v_i,v_j\rangle = c_{ij}$ are known, it is easy to evaluate the matrices of $I_1$, $I_2$ and $I_3$ in the basis $(v_1,v_2,v_3)$.
\begin{align}
I_1 &= 
\begin{pmatrix}
1 & 2c_{21} & 2c_{31}\\
0 & -1 & 0\\
0 & 0 & -1
\end{pmatrix}\\
I_2 &= 
\begin{pmatrix}
-1 & 0 & 0\\
2c_{12} & 1 & 2c_{32} \\
0 & 0 & -1
\end{pmatrix}\\
I_3 &= 
\begin{pmatrix}
-1 & 0 & 0 \\
0 & -1 & 0 \\
2 c_{13} & 2c_{23} & 1
\end{pmatrix}
\end{align}

\vskip10pt

We still have to see how  $p$, $q$, $r$ and $\theta$ determine $\langle v_i,v_j\rangle=c_{ij}$. 
For the time being, we suppose $r<\infty$.
In fact, we will prove that the matrix given by the $c_{ij}$'s is equal to:
\begin{equation}H=
\begin{pmatrix}
1 & \cos \frac \pi p & \cos \frac \pi r e^{\ii\theta} \\
\cos \frac \pi p & 1 & \cos \frac \pi q \\
\cos \frac \pi r e^{-\ii \theta} & \cos \frac \pi q & 1
\end{pmatrix}.
\end{equation}
And this shows that the $c_{ij}$'s fully determine $p$, $q$, $r$ and $\theta$ in return. 
This matrix is an Hermitian form preserved by $I_1$, $I_2$ and $I_3$. The determinant of this matrix is given by
\begin{equation}
1  + 2\cos(\theta)\cos\frac\pi p \cos\frac \pi q \cos\frac \pi r
- \cos\left(\frac \pi p\right)^2
- \cos\left(\frac \pi q\right)^2
- \cos\left(\frac \pi r\right)^2.
\end{equation}
This determinant allows us to decide when $H$ has $(2,1)$ for signature. Since the trace of $H$ is $3$, it implies that at least one eigenvalue is positive. Therefore, its determinant is negative if and only if $H$ has $(2,1)$ for signature. That is equivalent to:
\begin{equation}
\cos(\theta) < \frac{-1 + \cos\left(\frac \pi p\right)^2 + \cos\left(\frac \pi q\right)^2 + \cos\left(\frac \pi r\right)^2 }{2\cos\frac\pi p \cos\frac \pi q \cos\frac \pi r}.
\end{equation}
That must be the case since the original Hermitian form $\langle\cdot,\cdot\rangle$ has $(2,1)$ for signature.

If $p=2$ then $c_{12}$ vanishes and one can make both $c_{23}$ and $c_{13}$ real. Therefore, $(2,q,r)$  complex hyperbolic triangle groups are rigid: $\theta$ is always zero.

\vskip10pt 
We now justify the expression of $H$ by computing the $c_{ij}$'s.
Up to conjugation, we can suppose that $H_1\cap H_2$ is generated by $(0,0,1)$. This implies that $v_1$ and $v_2$ are both of the form $(x,y,0)$. Therefore, every $c_{ij}$ is given by $v_{i}^1\overline{v_{j}^1} + v_i^2\overline{v_j^2}$ since at least one of $v_i$ or $v_j$ has a vanishing third coordinate.

Therefore, geometrically speaking, $c_{ij}$ is the cosine of the angle in $\C^2$ formed by the complex lines generated by the first two coordinates of $v_i$ and $v_j$. It is the real part of $c_{ij}$ that is equal to the cosine of the angle formed by the vectors given by the first two coordinates of $v_i$ and $v_j$ (see~\cite[p. 36]{Goldman}). In other terms, if $\langle v_i,v_j\rangle$ is real and $v_i,v_j$ are unitary then $\langle v_i,v_j\rangle = \cos\theta$ with $\theta\in[0,\pi[$. If additionally $\langle v_i,v_j\rangle >0$ then $\theta\in[0,\pi/2]$ (that will be the case for us).

Note that $c_{13}$ is non real in general, but of course $\langle v_1,e^{\ii \theta}v_3\rangle = e^{-\ii\theta}\langle v_1,v_3\rangle = e^{-\ii\theta}c_{13}$ is real and can even be assumed to be positive.
The angle formed by $H_1$ and $H_2$ is equal to $\frac \pi p$ since $(I_1I_2)^p=e$. By taking the duals $v_1\in L_1$ and $v_2\in L_2$, we get $c_{12}=\cos\frac \pi p$. Likewise, $c_{23}=\cos\frac \pi q$ and $e^{-\ii\theta}c_{13}=\cos\frac \pi r$.

\vskip10pt
Finally, one can compute, with $i\neq j \neq k\neq i$:
\begin{equation}
\tr(I_iI_jI_kI_j) = 
16 |c_{ij}c_{kj}|^2 - 16\Re(c_{12}c_{23}c_{31})
+ 4|c_{ik}|^2 - 1,
\end{equation}
and note that in our conventions, we have $\Re(c_{12}c_{23}c_{31})=c_{12} c_{23} c_{13}\cos \theta$.
It shows that $\tr(I_iI_jI_kI_j)$ determines $\pm\theta$ once $(p,q,r)$ is known. Since the complex conjugation changes $\theta$ into $-\theta$, we deduce from our discussion the following results. In the case where $(p,q,r)=(3,3,r)$ we have in fact:
\begin{equation}
\tr(I_iI_jI_kI_j) = 4\cos\left(\frac \pi r\right)^2 - 4\cos\frac\pi r \cos\theta.
\end{equation}

\begin{proposition*}[\cite{Pratoussevitch}]\label{prop-unique-triangle}
Let $3\leq p \leq q \leq r < \infty$ be such that $\frac \pi p + \frac \pi q + \frac \pi r <\pi$. 
A representation of the triangle group
\begin{equation}
\Delta(p,q,r;\theta) = \left\langle I_1,I_2,I_3 \; \middle\vert \;
\begin{matrix} 
&I_1^2=I_2^2=I_3^2=e, \\
&(I_1I_2)^p=(I_2I_3)^q=(I_3I_1)^r=e
\end{matrix}\right\rangle
\end{equation}
into $\PU(2,1)$ is determined by $\theta = \arg(\langle v_1,v_2\rangle\langle v_2,v_3\rangle \langle v_1,v_3\rangle)$ up to conjugation, with $v_1,v_2,v_3$ as previously.
Up to conjugation and complex conjugation, it is determined by
$\tr(I_iI_jI_kI_j)$, with $i,j,k$ pairwise distinct.

Furthermore, $\theta$ verifies
\begin{equation}
\cos(\theta) < \frac{-1 + \cos\left(\frac \pi p\right)^2 + \cos\left(\frac \pi q\right)^2 + \cos\left(\frac \pi r\right)^2 }{2\cos\frac\pi p \cos\frac \pi q \cos\frac \pi r}
\end{equation}
and conversely, this condition suffices to define a representation with that value of $\theta$.
\end{proposition*}

The parameter $\theta$ can be taken in $[0,\pi/2]$ since the complex conjugation exchanges $\theta$ and $2\pi - \theta$. The possible values of $\tr(I_iI_jI_kI_j)$ are constrained by the preceding condition. For example, when $(p,q,r)=(3,3,r)$, we have
\begin{equation}
\cos\theta < \frac{-\frac 12 + \cos\left(\frac \pi r\right)^2 }{\frac 12 \cos\frac \pi r} = \frac{-1 + 2\cos\left(\frac \pi r\right)^2}{\cos\frac \pi r}.
\end{equation}
If we take a look at the trace of $I_iI_jI_kI_j$, its maximum is given for $\cos\theta$ minimal (that is to say $-1$) and its minimum by the maximum of $\cos\theta$. We have the inequalities:
\begin{align}
\tr(I_iI_jI_kI_j) &\leq 4\cos\frac\pi r\left(\cos\frac\pi r + 1\right),\\
\tr(I_iI_jI_kI_j) &>4\cos\left(\frac \pi r\right)^2 - 4\cos\frac\pi r \left(\frac{-1 + 2\cos\left(\frac \pi r\right)^2}{\cos\frac \pi r}\right) =4\left(1 - \cos\left(\frac \pi r\right)^2\right)>0.
\end{align}
This computation shows that the range of the values of $\tr(I_iI_jI_kI_j)$ is included in $\R_+$ and therefore, the range for which $\Delta(p,q,r;\theta)$ is discrete is of the form $[3,m]$, with $m$ the maximum stated before. The value $3$ is indeed reachable: the minimum value for $r$ is $4$, since we have to verify $\frac \pi p + \frac \pi q + \frac\pi r< \pi$, and the maximum value of $4(1 - \cos(\frac \pi r)^2)$ is indeed reached when $r$ is minimal. This value for $r=4$ is $1$.

One can compute the angular invariant required to have $I_iI_jI_kI_j$ unipotent. It is given by:
\begin{equation}
\cos\theta = \cos\frac \pi r - \frac 3{4\cos\frac \pi r}.
\end{equation}

\vskip10pt
When one or several of $p$, $q$ and $r$ are non finite, we can get similar results by replacing the undefined $c_{ij}$ with $\cosh(l_{ij}/2)$, where $l_{ij}$ is the distance between the two complex hyperbolic geodesics $H_i$ and $H_j$. See~\cite{Pratoussevitch}. In particular, it is still true that $\cos\theta$ is determined by $\tr(I_iI_jI_kI_j)$.

\section{Experimental approach}\label{sec-3}

In this section, we explain how we  approximate the limit sets of the representations appearing in the census~\cite{Ptolemy}. 

We also propose a comparative experiment by simulating the limit sets associated to the $(3,3,n)$ triangle groups. Since, up to conjugation and complex conjugation, there is a unique $\Delta(3,3,n;\theta_\infty)$ with $I_3I_2I_1I_2$ unipotent, the limit set is itself  unique up to translation  and complex conjugation.

 It will allow us to propose visual clues in order to distinguish the different $(3,3,n)$ triangle group fractals when $I_3I_2I_1I_2$ is unipotent.

The source code of the simulations and most of their results are available on the author's webpage~\cite{rph}. The code has been made open-source~\cite{rphgit}.

\subsection{Computing limit sets}

\begin{lemma}\label{lem-limset}
Let $\Gamma\subset \PU(n,1)$ be a subgroup. Let $\Gamma_L$ denote the subset of the loxodromic elements.
Suppose $\Gamma_L\neq \emptyset$.
Then the closure of the accumulation points for the  iteration dynamic:
\begin{equation}
\overline{\left\{x \, |\, \exists g\in \Gamma_L, \; \lim g^n x_0 = x\right\}} \text{ for any } x_0\in \mathbf H_{\mathbf C}^2,
\end{equation}
is equal to the full limit set $L(\Gamma)$.
\end{lemma}

\begin{proof}
Note that if $g$ is loxodromic then after any conjugation, $\gamma g \gamma^{-1}$ is again loxodromic. 

The set $A=\overline{\left\{x \, |\, \exists g\in \Gamma_L, \; \lim g^n x_0 = x\right\}}\subset\partial\mathbf H_\C^2$ is compact and has at least two points. By definition, $A\subset L(\Gamma)$. So we only need to show that $A$ is $\Gamma$-invariant, implying $L(\Gamma)\subset A$.

For any $\gamma\in \Gamma$, we have $\gamma (\lim g^n x_0 )= \lim (\gamma g \gamma^{-1})^n (\gamma x_0)$. But $\gamma g \gamma^{-1}$ is  loxodromic and its only attractive fixed point is the limit point for the orbits from both $x_0$ and $\gamma x_0$. Therefore $\lim (\gamma g \gamma^{-1})^n (\gamma x_0)=\lim (\gamma g \gamma^{-1})^n x_0 $ which shows the $\Gamma$-invariance. It implies $L(\Gamma)\subset A$ by minimality of $L(\Gamma)$.
\end{proof}

This lemma suggests a strategy to approximate $L(\Gamma)$: compute the attractive limit points of the loxodromic elements of $\Gamma$. 
However, this strategy requires to compute a very large number of elements $g\in \Gamma$. 
This can be done by  generating words of length $n$. If $\Gamma$ is described by two generators, then there are approximately $3^n$ words of length $n$. 

In practice, and this is particularly true with complex hyperbolic triangle groups, it is hard to get \emph{different} points from such a computation. 
One can often see large concentrations of points in tiny boxes and even many copies of the same point. 
This is partly due  to unknown relations between words, even at small length words.

Instead of only computing  words and getting their attractive limits, we used a second strategy in complement. When enough points are acquired, one can apply words on them (loxodromic or not) to get a better picture of the limit set. This method is much more efficient for it rarely makes redundant images. 
When nice symmetries are known (and for example, with complex hyperbolic triangles one knows the reflections $I_1$, $I_2$ and $I_3$), this allows a much better result. 

In practice, we first compute the attractive points of $n_1$-length words, then apply given symmetries on the set obtained, then apply $n_2$-length words  on them, and apply again symmetries. 

We show the different steps for two examples: one fractal  limit set and one dense limit set from representations of the fundamental group of m004 (the figure-eight knot's complement). See figures~\ref{fig-m004-1} and \ref{fig-m004-5} in the appendix.

\vskip10pt
Once this numerical approximation is done, we find the one giving fractals by looking if it is neither dense (high density of points) nor a smooth circle (easy to observe). The fractals we found are listed in the figures~\ref{fig-all-1} and~\ref{fig-all-2}.

\subsection{Complex hyperbolic $(3,3,n)$ triangle groups}

To compute the limit sets associated to $\PU(2,1)$-representations of $(3,3,n)$ triangle groups, we used our previous parametrization of the complex reflections $I_1$, $I_2$, and $I_3$. 
We picked the angular invariant for which $I_3I_2I_1I_2$ is unipotent.

For $n\in \{4,5,6,7\}$ we show three projections of the limit set and an additional diagram proposing a visual clue to recognize the limit set, see figures \ref{fig-2-1}, \ref{fig-2-2}, \ref{fig-2-3}, and \ref{fig-2-4} in the appendix. This visual clue consists in looking for a pair of symmetric spikes and inspect the middle. We count  the largest outer holes. When $n=4$ there is none, when $n=5$ there is one, when $n=6$ there are two, when $n=7$ there are three, etc.

\section{Morphisms and redundancy}\label{sec-4}

From the census of the boundary unipotent representations in~\cite{FKR} and stored in SnapPy~\cite{Ptolemy}, we numerically approximated all their limit sets. After a visual inspection, we kept the representations that gave fractals (see~\cite{rph} for the numerical results and the figures~\ref{fig-all-1} and~\ref{fig-all-2}).

Those representations come in pairs by complex conjugation of the coefficients. In this section, we study many of them by a classification into $(3,3,n)$ triangle groups. $19$ pairs will be directly classified and $4$ more will be  constructed. The classification will be systematic, following a method that we will  describe.

\paragraph{Notes on the selection}
It is already known that m004-1 and m004-3 are related by the composition of a figure-eight knot's symmetry  (see~\cite{DF}). Therefore we only classify one of the two with the systematic procedure.

A representation of m038 presents the characteristics of a $(3,4,4)$ complex hyperbolic triangle group. Indeed, m038 has such a representation  according to a preprint of Ma and Xie~\cite{MaXie} that the author has been able to consult.
A representation  of m137 presents the characteristics of a $(3,4,5)$ complex hyperbolic triangle group.

\vskip10pt
We  summarize the mathematical result  in the following theorem. It is fully independent from the numerical computations. The method of the proof  will be explained later.

Recall that $\Delta_2(3,3,n;\theta_\infty)$ denotes the even-length subgroup of $\Delta(3,3,n;\theta_\infty)$, where the angular invariant $\theta_\infty$ is chosen so that $I_3I_2I_1I_2$ is unipotent.

\begin{theorem}\label{thm-rest}
\textthmrest
\end{theorem}

\experimentaltable

The manifolds m039 and s000 do not appear in the census of~\cite{FKR} for there are respectively described  by five and six tetrahedra. But the result still applies and we will construct their triangular representations with boundary  unipotent holonomy from their fundamental groups. We found those manifolds by applying the theorem of Acosta (see below).

The manifold m053 will be treated differently than the other members of the census. The initial description in SnapPy of its two triangular representations m053-1 and m053-7 are quite hard to use. We proceed the computation of another description (with the help of the method {\tt Manifold.randomize()} of SnapPy).

Note that for each row, at most one representation is a uniformization of the corresponding manifold.  Deraux~\cite{Deraux} encountered the same phenomenon  with the manifolds m009 and m015.
In fact, we identify the uniformizations with the following result of Acosta.

\begin{theorem*}[\cite{Acosta}]
Let $4\leq n \leq \infty$. The manifold at infinity of $\mathbf H_{\mathbf C}^2 / \Delta_2(3,3,n;\theta_\infty)$ is the Dehn filling with slope $(1,n-3)$ of any cusp of the Whitehead link complement.
\end{theorem*}

With SnapPy, it is possible to compute Dehn surgeries on the Whitehead link complement. Note that m129 is the Whitehead link complement in our table. One has to be careful: the marking of the peripheral holonomy is not unique, with SnapPy one needs to call the manifold $5^2_1$ and fill a cusp with the meridian equal to $n-3$ and the longitude equal to $1$. 

This procedure gives the selection of the uniformized manifold along each row.  Note that we do have a true CR uniformization since Acosta did construct the CR structures by CR surgery. (It is not only a topological result.)

The first uniformization of m129 (the Whitehead link complement) was  shown by Schwartz~\cite{Schwartz2}, but the present uniformization by a $(3,3,\infty)$ triangle group with boundary unipotent holonomy was studied by Parker and Will~\cite{ParkerWill}.

\vskip10pt
We now explain how we prove the rest of theorem  \ref{thm-rest}. We start with the following fundamental result about $\PU(2,1)$ subgroups generated by two elements. It is the restricted case of a more general result in ${\rm SL}_3(\C)$.

\begin{proposition}[\cite{Lawton,Will09}]\label{prop-Lawton}
If $\Gamma$ is a group generated by two elements $a$ and $b$, then any irreducible representation $\xi\colon\Gamma\to\PU(2,1)$ is fully determined up to conjugation in $\PU(2,1)$  by $\tr(\xi(a)),\tr(\xi(b)),\tr(\xi(ab)),\tr(\xi(ab^{-1}))$ and $\tr(\xi([a,b]))$.
\end{proposition}

In fact, in $\PU(2,1)$, $\tr(\xi([a,b]))$ is a root of a second degree polynomial with real coefficients and the two roots are  either equal or complex conjugates.

With
\begin{align}
\Lambda_2(3,3,n) = \left\langle x,y \; \middle\vert \; x^3=y^3=(xy)^n=e\right\rangle,
\end{align}
recall that we have a surjective morphism $\Lambda_2(3,3,n)\to\Delta_2(3,3,n;\theta_\infty)$ given by $(x,y)\mapsto (I_1I_2,I_2I_3)$. The traces of the images of $x,y,xy$ and $xy^{-1}$ are all real and prescribed by the data $(3,3,n)$. (When $n=\infty$, we send $xy$ to a unipotent transformation.) Therefore, by doing a complex conjugation on the representation, we only conjugate the trace of the image of $[x,y]$.

\begin{corollary}\label{cor-identif}
Let $\pi_1(M)$ be a fundamental group generated by two elements $a$ and $b$. 
Let  $\xi\colon\pi_1(M)\to \PU(2,1)$ be a representation.
Let $\rho\colon\pi_1(M)\to\Lambda_2(3,3,n)$ be a surjective morphism and 
 $\gamma\colon\pi_1(M)\to\Delta_2(3,3,n;\theta_\infty)$ be the composition of $\rho$ with $\Lambda_2(3,3,n)\to\Delta_2(3,3,n;\theta_\infty)$.
 
 Let  $w_x,w_y,w_{xy},w_{xy^{-1}}\in\pi_1(M)$ be  such that $\rho(w_x)=x,\rho(w_y)=y$ and $w_{xy}=w_xw_y$, $w_{xy^{-1}}=w_xw_y^{-1}$.
We make the following hypotheses:
\begin{enumerate}
\item The traces $\tr\circ\xi$ and $\tr\circ \gamma$ coincide on the words $w_x,w_y,w_{xy}$ and $w_{xy^{-1}}$.
\item There exists $w_a,w_b\in\langle w_x,w_y\rangle\subset\pi_1(M)$ such that $\xi(w_a)=\xi(a)$, $\xi(w_b)=\xi(b)$ and  $\gamma(w_a)=\gamma(a)$, $\gamma(w_b)=\gamma(b)$.
\end{enumerate}

Then $\gamma$ and $\xi$ are equal up to conjugation and  complex conjugation.
\end{corollary}

\begin{proof}
Given the notations, the first assumption shows that $\xi$ and $\gamma$ are equal up to complex conjugation and conjugation on $\langle w_x,w_y\rangle\subset\pi_1(M)$.

Up to the conjugations, assume  that $\xi=\gamma$ on $\langle w_x,w_y\rangle$. Then by the second hypothesis $\xi(a)=\xi(w_a)=\gamma(w_a)=\gamma(a)$ and $\xi(b)=\xi(w_b)=\gamma(w_b)=\gamma(b)$. Therefore, $\xi=\gamma$ on $\pi_1(M)$.
\end{proof}

Recall that the trace of a transformation $w$ determines its type.
If a transformation $w$ is elliptic with order $k$, then $\tr(w)=4\cos(\pi/k)^2-1$. The transformation $w$ is unipotent if, and only if, $\tr(w)=3$ and $w\neq {\rm id}$.

\paragraph{The method}This is how we will identify all our boundary unipotent representations. An application of this method is detailed in the next section.
\begin{enumerate}
\item We describe the relation of $\pi_1(M)$ as consequence of several simpler relations “of triangle type”. Those simpler relations are verified by the representation in the census (but not by $\pi_1(M)$). 

For example, $a^2ba^{-1}b^{-2}a^{-1}ba$ is implied by $a^4$, $b^3$ and $(a^{-1}b)^3$.
It is  the most challenging step and also the most important. It is in general very difficult to find   such simpler relations.

This step is inscribed in the first table. The “id” column allows us to recover the representation from SnapPy. Relations can be checked with the SageMath code in the appendix.
\item We construct the surjective morphism $\rho$ and recover the values of the various words $w_{x},w_y,w_{xy}$ and $w_{xy^{-1}}$.  We use the previous relations “of triangle type” to check that we have a morphism. This step is inscribed in the second and third tables. It is a completely autonomous step from all the computations. It only involves $\pi_1(M)$ and by hand  computations.
\item We find values of $w_a$ and $w_b$ and  it can be checked by hand that they verify the required equalities for both $\xi$ and $\gamma$. Those words are inscribed in the third table.
\item We verify that the traces of the four elements match with the evidence of the first step and by exact computations, \emph{e.g.}  for $w_{xy^{-1}}$ (see the SageMath code in the appendix for this).
\end{enumerate}

Those four steps prove the identification we are looking for. But we add a last table with the description of the peripheral holonomy. The reason is that one can only consider the second step (doable by hand) and with the last table show that the morphism constructed has boundary unipotent holonomy, independently from the census. (One should play with various the conjugations such as $y(xy^{-1})y^{-1} = yxy$ and $y^{-1}(xy^{-1})y=y^{-1}x$ in $\Lambda_2(3,3,n)$.)

For the construction of this last table, we only use the relations described in the first step and implied by the description of the morphism. 

\begin{proof}[Proof of theorem \ref{thm-rest}]
Let $M$ be a manifold in the table and $\Delta_2(3,3,n;\theta_\infty)$ an assigned group. By the second step, we construct a morphism $\rho\colon\pi_1(M) \to \Lambda_2(3,3,n)$. This morphism gives a representation $\xi$ by the composition $\pi_1(M)\to\Lambda_2(3,3,n)\to\Delta_2(3,3,n;\theta_\infty)$.
Finally, $\xi$ has boundary unipotent holonomy by the last table.
\end{proof}

\paragraph{On the notations}We keep the notations of the corollary and write $A,B,X,Y$ instead of $a^{-1},b^{-1},x^{-1},y^{-1}$ respectively.

\subsection{$\Lambda_2(3, 3, 4)$ -- m004, m022, m029, m034, m081 and m117}\label{subsec-334}

We survey the method for m004-1. We denote this representation by $\xi\colon\pi_1({\rm m004})\to\PU(2,1)$. By SnapPy, the fundamental group of m004 is presented by the generators $a$ and $b$ and the relation $a^2bAB^2Aba$. 

In the census, $\xi$ is a representation verifying (by exact computation) the relations $\xi(a)^4=e$, $\xi(b)^3=e$ and $\xi(Ab)^3=e$. Those three relations  alone imply the one of the fundamental group: $(a^2)bA(B^2)Aba = (AA)bA(b)Aba= A (Ab)^3 a =e$.

Those first datas are reported in the first table.
They suggest  that we should seek for a $(3,3,4)$ triangle group representation.

\vskip10pt
We define a new morphism $\rho\colon\pi_1({\rm m004})\to\Lambda_2(3,3,4)$ by $\rho(a)= xy$ and $\rho(b)= y$. To check that this is indeed a morphism, it suffices to check the relations $a^4,b^3,(Ab)^3$. And that is the case: in $\Lambda_2(3,3,4)$, $xy$ has order $4$, $y$ has order $3$ and $Ab$ is sent on $Yxy$ that has order $3$ since $x$ has order $3$. This is given in the second table.

With the data of the third table, we show that we have constructed a \emph{surjective} morphism $\rho\colon\pi_1({\rm m004})\to\Lambda_2(3,3,4)$. Since $\Lambda_2(3,3,4)$ is generated by $x$ and $y$, it suffices to give antecedents to both. The word $w_x=aB$ in $\pi_1({\rm m004})$ is sent to $xyY=x$ and the word $w_y=b$ is sent to $y$.

\vskip10pt
Now we prove that $\gamma$ (which is the composition of $\rho$ with $\Lambda_2(3,3,4)\to\Delta_2(3,3,4;\theta_\infty)$) and $\xi$ have equal images up to conjugation and complex conjugation.
According to corollary~\ref{cor-identif}, it suffices to find the words $w_{x},w_y,w_{xy},w_{xy^{-1}},w_a$ and $w_b$. This is all given in the third table and verifications are straightforward.

The traces of $w_x,w_y,w_{xy}$ will coincide since they verify the triangular relations in both representation $\gamma$ and $\xi$. With the word $w_{xy^{-1}}$ we check by formal computation that the trace under $\xi$ is indeed $3$.

\vskip10pt
If one does not want to involve the representation $\xi$ of the census, but only construct the morphism $\rho$ \emph{a priori}, the last table allows us to check that the boundary is unipotent under $\rho$ composed with the usual map $\Lambda_2(3,3,4)\to\Delta_2(3,3,4;\theta_\infty)$. Here, the boundary  is described by the two words $ab$ and $aBAbABab$. Under $\rho$, we have $\rho(aBAbABab) = \rho(ab)^3$ and $\rho(ab)=xyy=xY$ that is indeed mapped to a unipotent element of $\PU(2,1)$.

\begin{center}
\begin{tabular}{|ccc|c|}
\hline
 & id & $\pi_1$ relation&  Relations\\
\hline
m004-1 & [0,0] & $a^2bAB^2Aba$ &  $a^4$, $b^3$, $(Ab)^3$
\\
m022-1 & [0,0] & $ab^5abA^2b$ &  $a^3$, $b^4$, $(ab)^3$
\\
m029-1 & [0,0] & $aBab^3A^2b^3$ &  $a^3$, $b^4$, $(aB)^3$
\\
m034-1 & [0,0] & $a^3b^2ABAb^2$  &  $a^4$, $b^3$, $(AB)^3$
\\
m081-1 & [0,0] & $ab^3aBa^4B$ & $a^3$, $b^4$, $(aB)^3$ 
\\
m117-1 & [0,0] & $a^2b^2a^2b^2ABAb^2$ &  $a^3$, $b^3$  $(AB)^4$
\\
\hline
\end{tabular}

\begin{tabular}{|c|cc|}
\hline
 &  $a$ & $b$  \\
\hline
m004 & $xy$ & $y$ 
\\
m022 & $y$ & $(xy)^{-1}$ 
\\
m029 & $y$ & $xy$ 
\\
m034 & $(xy)^{-1}$ & $x$ 
\\
m081 & $x^{-1}$ & $(xy)^{-1}$ 
\\
m117 & $x^{-1}$ & $y^{-1}$  
\\
\hline
\end{tabular}

\begin{tabular}{|c|cccc|cc|}
\hline
 & $w_x$ &$w_y$ &$w_{xy}$ & $w_{xy^{-1}}$ &$w_a$ & $w_b$ \\
\hline
m004 
& $aB$ & $b$ & $a$ & $aB^2$ &$w_{xy}$ & $w_y$
\\
m022 
&  $BA$ & $a$ & $B$ & $BA^2$ & $w_y$ &$w_{xy}^{-1}$
\\
m029 
& $bA$ & $a$ & $b$ & $bA^2$ &$w_y$ & $w_{xy}$
\\
m034 
& $b$ & $BA$ & $A$ & $bab$ & $w_{xy}^{-1}$ &$w_x$
\\
m081 
& $A$ & $aB$ &$B$& $Ba^2$ & $w_x^{-1}$ & $w_{xy}^{-1}$
\\
m117 
&  $A$ & $B$ & $AB$ & $Ab$ &$w_x^{-1}$ & $w_y^{-1}$
\\
\hline
\end{tabular}

\begin{tabular}{|c|cc|}
\hline
&Peripheral& holonomy\\
\hline
m004 & $ab$ & $aBAbABab\equiv (ab)^3$
\\
m022 & $Ba$ & $A^2babA\equiv Ba$
\\
m029 & $ab^2$ & $bA^3b^3\equiv e$
\\
m034 & $b^2a^2\equiv BA^2$ & $A^3B^3A\equiv e$
\\
m081 & $b^2a\equiv B^2a$ &  $Ba^3Ba\equiv B^2a$
\\
m117 & $bA$ & $BA^3BA\equiv bA$
\\
\hline
\end{tabular}
\end{center}

\subsection{$\Lambda_2(3, 3, 5)$ -- m009, m015, m035, m142 and m146}

\begin{center}
\begin{tabular}{|ccc|c|}
\hline
 & id & $\pi_1$ relation &  Relations\\
\hline
m009-1 & [0,1] &$ a^2bABa^2BAb$
& $a^5$ ,  $(a^2B)^3$, $(a^2b)^3$
\\
m015-2 & [0,1]& $ab^2A^2b^2aB^3$
 & $a^3$, $b^5$, $(abb)^3$
\\
m035-5 & [0,3]& $ab^3A^2b^3aB^2$
 & $a^3$, $b^5$, $(aBB)^3$
\\
m142-1 & [0,0]&$ab^2aBab^2aBa^4B$
 &  $a^3$, $b^3$, $(aB)^5$
\\
m146-3 & [0,3] & $a^2b^2a^3b^2a^2BAB$
& $a^3$, $b^5$, $(AB)^3$
\\
\hline
\end{tabular}

\begin{tabular}{|c|cc|}
\hline
 &  $a$ & $b$  \\
\hline
m009 & $xy$ & $(YX)^{2}Y$ 
\\
m015 & $Y$ & $(YX)^2$
\\
m035 & $y$ & $(YX)^2$
\\
m142 & $YXy$ & $y$  
\\
m146 & $y$ & $YX$ 
\\
\hline
\end{tabular}

\begin{tabular}{|c|cccc|cc|}
\hline
  & $w_x$ &$w_y$ &$w_{xy}$ & $w_{xy^{-1}}$ &$w_a$ & $w_b$\\
\hline
m009 
& $a^3b$ & $BA^2$ & $a$ & $a^3ba^2b$
& $w_{xy}$ & $w_{xy}^{-2}w_y^{-1}$
\\
m015 
& $b^2a$ & $A$ & $b^2$ & $b^2A$
&$w_y^{-1}$ & $w_{xy}^{-2}$
\\
m035 
&$b^2A$ & $a$ & $b^2$ &$b^2a$ 
&$w_y$ & $w_{xy}^{-2}$
\\
m142
& $bAB$ & $b$ & $bA$ &$bAB^2$
& $w_{xy}^{-1}w_y$ & $w_y$ 
\\
m146 
& $BA$ & $a$ & $B$ & $BA^2$
& $w_y$ & $ w_{xy}^{-1}$
\\
\hline
\end{tabular}

\begin{tabular}{|c|cc|}
\hline
 & Peripheral & holonomy  \\
\hline
m009 & $ab$ &  $ABa^3BAb\equiv (ab)^2$ 
\\
\hline
m015 & $bA$ &  $ab^2A^3b^2\equiv (bA)^{-1}$
\\
\hline
m035 & $ab$ & $ab^3A^3b^3\equiv ab$
\\
\hline
m142 & $BA$ & $bA^3bA\equiv BA$
\\
\hline
m146 & $ba^2\equiv bA$ &  $BABAB^2\equiv aB$
\\
\hline
\end{tabular}
\end{center}

\subsection{$\Lambda_2(3, 3, 6)$ -- m023}

\begin{center}
\begin{tabular}{|ccc|c|}
\hline
 & id & $\pi_1$ relation & Relations\\
m023-1 & [0,0]  & $aBAb^2ABab^3$ &  $b^6$, $(Abb)^3$, $(aB^3)^3$
\\
\hline
\end{tabular}

\begin{tabular}{|c|cc|}
\hline
 &  $a$ & $b$  \\
\hline
m023 & $(YX)^{2}Y^2X$  & $xy$
\\
\hline
\end{tabular}

\begin{tabular}{|c|cccc|cc|}
\hline
 & $w_x$ & $w_y$ & $w_{xy}$ & $w_{xy^{-1}}$&$w_a$&$w_b$ \\
\hline
m023 & $b^3ab$  & $BAB^2$ & $b$ & $b^3ab^3ab$
& $w_{xy}^{-3}w_{x}w_{xy}^{-1}$& $w_{xy}$
\\
\hline
\end{tabular}

\begin{tabular}{|c|cc|}
\hline
 &  Peripheral & holonomy \\
\hline
m023  & $bab$  & $b^2aBABab^2 \equiv (BAB)^2$
\\ 
\hline
\end{tabular}
\end{center}

\subsection{$\Lambda_2(3, 3, 7)$ -- m032 and m045}

\begin{center}
\begin{tabular}{|ccc|c|}
\hline
 & id &$\pi_1$ relation & Relations\\
\hline
m032-7 & [0,2] &  $a^2B^2Ab^5AB^2$ & $a^3$, $b^7$, $(AB^2)^3$
\\
m045-8 & [0,4] & $a^3b^2a^3BA^4B$ &  $a^7$, $b^3$, $(a^3B)^3$
\\
\hline
\end{tabular}

\begin{tabular}{|c|cc|}
\hline
 &  $a$ & $b$\\
\hline
m032 & $x$ & $(xy)^3$  
\\
m045 & $(xy)^2$  & $X$
\\
\hline
\end{tabular}

\begin{tabular}{|c|cccc|cc|}
\hline
& $w_x$ & $w_y$ & $w_{xy}$ & $w_{xy^{-1}}$&$w_a$&$w_b$\\
\hline
m032 
& $a$ & $AB^2$ & $B^2$ & $ab^2a$
& $w_x$ & $w_{xy}^3$
\\
m045 
& $B$ & $bA^3$ & $A^3$ & $Ba^3B$
& $w_{xy}^2$ & $w_{x}^{-1}$
\\
\hline
\end{tabular}

\begin{tabular}{|c|cc|}
\hline
 &  Peripheral  & holonomy \\
\hline
m032 & $B^3a$ & $b^2A^3b^2a \equiv B^3a$
\\
\hline
m045 & $AB$  & $bA^3B^3A^3 \equiv ba $
\\
\hline
\end{tabular}
\end{center}

\subsection{$\Lambda_2(3, 3, \infty)$ -- m129 and m203}

With the complex hyperbolic triangle group $\Delta_2(3,3,\infty;\theta_\infty)$ we ask $I_1I_3$ to be unipotent: we  send $xy$ to a unipotent transformation.

\begin{center}
\begin{tabular}{|ccc|c|}
\hline
 & id & $\pi_1$ relation & Relations\\
\hline
m129-1 & [0,0] & $a^3B^2abA^3b^2AB$ & $a^3$, $b^3$
\\
m203-1 & [0,0] &$ a^3b^2a^2BA^3B^2A^2b$ & $a^3$, $b^3$ 
\\
\hline
\end{tabular}

\begin{tabular}{|c|cc|}
\hline
 &  $a$ & $b$ \\
\hline
m129 &  $xyx^{-1}$ & $x$ 
\\
m203 & $xyx^{-1}$ & $x$
\\
\hline
\end{tabular}

\begin{tabular}{|c|cccc|cc|}
\hline
& $w_x$ & $w_y$ & $w_{xy}$ & $w_{xy^{-1}}$&$w_a$&$w_b$\\
\hline
m129 
& $b$ & $Bab$ & $ab$ &$Ab$
&$w_{xy}w_x^{-1}$ & $w_x$
\\
m203 
& $b$ & $Bab$ & $ab$ &$Ab$
&$w_{xy}w_x^{-1}$ & $w_x$
\\
\hline
\end{tabular}

\begin{tabular}{|c|cc|}
\hline
 &  Peripheral  & holonomy \\
\hline
m129 & $A^2b\equiv ab$ & $ A^3b^2A \equiv BA$
 \\ &  $Ab$ &  $bA^3ba \equiv Ba$
\\
\hline
m203 & $a^2b\equiv Ab$ & $B^2A^3B\equiv e$ 
\\ &  $ab$  & $BA^3B^2A^3\equiv e$
\\ 
\hline
\end{tabular}
\end{center}

\subsection{m039, s000 and m053}

In this final section, we construct the remaining representations with boundary unipotent holonomy. 
We were not able to identify them with representations from the census for the following reasons. 

For m039 and s000 this comes from the fact that they are described by respectively five and six tetrahedra and therefore are not part of the census. For m053, the initial description provided is too complicated to be used out of the box. But with an other description, one can still give  triangular representations and the identification with m053-1 and m053-7 follows from the visual clues.

For those representations, we follow the same method to construct the representations, with the exception of the experimental verification of the traces and relations that are not necessary here.

\begin{center}
\begin{tabular}{|cc|c|}
\hline
 & $\pi_1$ relation &  Relations
\\
\hline
m039 & $a^6BAb^2AB$  & $a^7$, $b^3$, $(AB)^3$ \\
m053 & $a^3b^2a^3BA^5B$ & $a^4$, $b^3$, $(AB)^3$ \\
&& $a^8$, $b^3$, $(a^3B)^3$\\
s000 & $a^7BAb^2AB$ &  $a^8$, $b^3$, $(AB)^3$  \\
\hline 
\end{tabular}

\begin{tabular}{|c|c|cc|}
\hline
 & Triangle & $a$ & $b$  
\\
\hline
m039 & $\Lambda_2(3,3,7)$ &  $YX$  & $y$ 
 \\
m053 & $\Lambda_2(3,3,4)$ &  $xy$ & $X$ 
\\
m053 & $\Lambda_2(3,3,8)$ &  $(xy)^3$ & $x$
\\
s000 & $\Lambda_2(3,3,8)$ &  $YX$ &$y$ 
\\
\hline
\end{tabular}

\begin{tabular}{|c|c|cccc|}
\hline
 &  Triangle
&$w_x$ & $w_y$& $w_{xy}$ & $w_{xy^{-1}}$ 
\\
\hline
m039 
& $\Lambda_2(3,3,7)$ &  $AB$ & $b$ &$A$ &$bA$ 
 \\
 m053 
& $\Lambda_2(3,3,4)$ & $B$ & $ba$ &$a$& $BAB$
\\
m053 
&$\Lambda_2(3,3,8)$ &  $b$ & $Ba^3$ & $a^3$ & $bA^3b$
\\
s000 
& $\Lambda_2(3,3,8)$ &  $AB$ & $b$ & $A$ & $Ab$ 
  \\
\hline
\end{tabular}

\begin{tabular}{|c|cc|}
\hline
& Peripheral & holonomy\\
\hline
m039 & $A^5b\equiv a^2b$ & $aB^3ab\equiv a^2b$ \\
\hline
m053 & $BA^3$ & $B^2A^3bA^3B$\\  
with $\Lambda_2(3,3,4)$ & $\equiv Ba$ & $\equiv Ab$\\
with $\Lambda_2(3,3,8)$ & $\equiv BA^3$ & $\equiv a^3b $ \\ 
\hline
s000 & $bA$ & $BabaB^2\equiv bA$ \\
\hline
\end{tabular}
\end{center}

\section{More diverse examples}\label{sec-5}

To study more examples, we note the following.
Let $M$ be a hyperbolic manifold with one cusp.
 If a representation $\pi_1(M)\to\PU(2,1)$ has for boundary  a parabolic subgroup \emph{of rank one} (generated by a single element), then there exists a relation between the two generators of its peripheral holonomy. This relation can be used to factorize the morphism through a Dehn filling.
 
 Now, a manifold on which we executed a Dehn filling will have a different presentation. Often, this presentation enables to see more rapidly the triangle group involved if we are in the case of theorem~\ref{thm-rest}.

\subsection{A homological obstruction}
The existence of a \emph{surjective} morphism (as in theorem~\ref{thm-rest}) is submitted to an easy  criterium.

\begin{lemma}
If there exists a surjective morphism $\rho\colon \pi_1(M)\to\Lambda_2(p,q,r)$ then there exists a surjective morphism $\widetilde\rho\colon H_1(M;\Z) \to \Lambda_2(p,q,r)^{\rm ab}$, where we took the abelianizations.\qed
\end{lemma}

Among the representations not covered in theorem~\ref{thm-rest}, there are three that have a fractal limit set corresponding to a $\Delta_2(3,3,n;\theta_\infty)$.

The representation m045-1 shows the sign of a $\Delta_2(3,3,4;\theta_\infty)$ group as image. See its approximated limit set  (figure \ref{fig-m045}) and compare with the visual clues (figure \ref{fig-2-1}) in the appendix. 

The representations m035-1 and m130-1 both resemble to $\Delta_2(3,3,\infty;\theta_\infty)$. Compare figures~\ref{fig-3_3_inf} and \ref{fig-m035}. See also~\cite{rph} for the numerical approximations in a three-dimensional display.

\begin{proposition}\label{prop-counterex}
The boundary unipotent representations m045-1, m035-1 and m130-1  do not verify theorem \ref{thm-rest}. That is to say, those representations are not surjective morphisms of the form $\rho\colon\pi_1(M)\to\Lambda_2(3,3,n)$.
\end{proposition}

In the case of $\Lambda_2(3,3,n)$, the abelianization of this abstract group is $\Z/3\Z$ or $\Z/3\Z\oplus \Z/3\Z$ depending on the class of $n$ modulo $3$. For $\Lambda_2(3,3,4)$ we get $\Z/3\Z$, and for $\Lambda_2(3,3,\infty)$ we get $\Z/3\Z \oplus \Z/3\Z$.

\begin{proof}[Proof of proposition \ref{prop-counterex}]
The representation m045-1 (with id  [0,0]) has  boundary unipotent holonomy of rank one. Therefore, with a relation between the two peripheral curves. 

This relation gives a factorization by the corresponding Dehn filling. In fact, in the terms of SnapPy and the data of m045-1 in~\cite{Ptolemy}, we get the Dehn filling ${\rm m045}(5,1)$. This can be verified with exact computations with the code in the appendix.

Computations of SnapPy show $H_1({\rm m045}(5,1)) = \Z / 14\Z$. But $14$ and $3$ are relatively prime  and therefore there exists no surjective morphism from $H_1({\rm m045}(5,1))$ to $\Z/3\Z$, giving an impossibility with the previous lemma.

\vskip10pt
Similarly, the representations m035-1 (of id [0,0]) and m130-1 (of id [0,1])  factorize through the Dehn fillings ${\rm m035}(3,-1)$ and ${\rm m130}(2,-1)$  respectively.

Now, $H_1({\rm m035}(3,-1)) = \Z/20\Z $ and $H_1({\rm m130}(2,-1)) = \Z/16\Z$. Thus there can't be any surjective morphism to $\Z/3\Z\oplus \Z/3\Z$, raising an impossibility with the lemma.
\end{proof}

\vskip10pt
The fact that m045-1 has a limit set so similar to the one of $\Delta_2(3,3,4;\theta_\infty)$ suggests the existence of a close relationship. For instance, there  might exist a morphism with finite index image in $\Delta_2(3,3,4;\theta_\infty)$. But this remains an open question.

\subsection{A Lagrangian triangle group}

As we will see (proposition~\ref{prop-m023}), there exists at least one representation with boundary unipotent holonomy,
\begin{equation}
\xi \colon \pi_1(M)\to\Lambda_2(p,q,r)\to\PU(2,1),
\end{equation}
surjective on $\Lambda_2(p,q,r)$ but \emph{that cannot be factorized} by 
\begin{equation}
\pi_1(M)\to\Lambda_2(p,q,r)\to\Delta_2(p,q,r;\theta)\subset \PU(2,1).
\end{equation}

\begin{lemma}
For any $4\leq n<\infty$, let $\eta\colon \Lambda_2(3,3,n)\to \PU(2,1)$ be a representation with $\eta([x,y])$ unipotent. Then  $\tr(\eta(xy^{-1}))$ is not real.
\end{lemma}

\begin{proof}
We use the exact parametrization of Guilloux and Will~\cite{GuillouxWill}. 
The formal computations that we will described have too many terms to be reproduced here, but they can be recovered in the code of the author~\cite{rphgit} or directly reproduced with the description of~\cite{GuillouxWill}.

Recall with proposition~\ref{prop-Lawton} that the discriminant $\Delta$ of the equation for which $\tr(\eta([x,y]))$ is solution is always negative or null.
Since we want $\tr(\eta([x,y]))=3$, we must have $\Delta=0$. This equation can be solved by exact computations and is only a function of $\tr(\eta(xy^{-1}))$ since $\tr(\eta(x)),\tr(\eta(y))$ and $\tr(\eta(xy))$ are determined by their elliptical order.

With the resolution of $\Delta=0$ in terms of the only free variable $\tr(\eta(xy^{-1}))$, one can compute  $\tr(\eta([x,y]))$ explicitly and solve the equation $\tr(\eta([x,y]))=3$ in terms of $\tr(\eta(xy^{-1}))$. It is again a quadratic equation. There are in general two  solutions  and are complex conjugates. We take the solution with positive imaginary part.

Numerical approximations show that 
for any $n<\infty$, the imaginary part of $\tr(\eta([x,y]))$ is always non zero. For example, when $n=4$, $\tr(\eta(xy^{-1}))\simeq 2.820 + \ii \cdot 0.222$.
\end{proof}

The representation m023-7 (with id [0,4]) of the census has a limit set different from the usual triangular subgroups, see figure~\ref{fig-m023-7}. The boundary holonomy has rank one, and therefore m023-7 factorizes through a Dehn filling. In this case, it is ${\rm m023}(2,-1)$.

\begin{proposition}\label{prop-m023}
There exists a surjective morphism 
\begin{equation}
\rho\colon \pi_1({\rm m023}) \to \Lambda_2(3,3,4),
\end{equation}
and a composition $\pi_1({\rm m023})\to\Lambda_2(3,3,4)\to\PU(2,1)$ having boundary unipotent holonomy.
The boundary unipotent representation in $\PU(2,1)$ is not induced by any representation $\Lambda_2(3,3,4)\to\Delta_2(3,3,4;\theta)$.
\end{proposition}

Note that any inclusion $\Lambda_2(p,q,r)\to\Delta_2(p,q,r;\theta)$ sends $xy^{-1}$ to a matrix with real trace in $\PU(2,1)$. Indeed, when described in terms of $I_1,I_2$ and $I_3$, we have $xy^{-1} \mapsto I_1I_2I_3I_2 = I_1(I_2I_3I_2)$ and this is the product of two reflections.

\begin{proof}
One of the presentations of $\pi_1({\rm m023}(2,-1))$ given by SnapPy (using also {\tt M.randomize()}) is 
\begin{equation}
\pi_1({\rm m023}(2,-1)) = \{a,b \, | \, a^4b^3=aBab^2ab^2=e\}
\end{equation}
with the peripheral curves given by $aBAb$ and $aBAbaBAb$.
Now, we define
\begin{equation}
\rho(a) = xy, \; \rho(b) = x
\end{equation}
and one can verify the relations of the fundamental group but also  $(aBAb)^2 = aBAbaBAb$.

The boundary holonomy of this representation is prescribed by $aBAb$ and this word has for image  $xyXYXx = xyXY = [x,y]$. We conclude with the previous lemma.
\end{proof}

Now we can compare the fractal of m023-7 with the one of $\Lambda_2(3,3,4)\to\PU(2,1)$ sending $[x,y]$ to a unipotent element as in the lemma (figure~\ref{fig-3_3_4-lag}). Visually, they match indeed.
 
 \vskip10pt

Denote $\rho\colon\pi_1({\rm m023})\to \Lambda_2(3,3,4)$ the surjective morphism from the proposition. Denote $\xi\colon\pi_1({\rm m023})\to\PU(2,1)$ its composition with $\eta\colon \Lambda_2(3,3,4)\to\PU(2,1)$ sending $[x,y]$ to a unipotent element.

A consequence of the fact that $\tr(\eta([x,y]))=3$ is that $\xi(\pi_1({\rm m023}))=\langle\eta(x),\eta(y)\rangle$ is $\R$-decomposable: it is generated by three antiholomorphic reflections.  This comes from a theorem of Paupert and Will:
\begin{theorem*}{\cite{Paupert}}
Let $A, B \in \PU(2,1)$ be two isometries not ﬁxing a common point in $\mathbf H^2_\C$ . Then the pair $A, B$ is $\R$-decomposable if and only if the commutator $[A,B]$ has a ﬁxed point in $\mathbf H_\C^2$ whose associated eigenvalue is real and positive.
\end{theorem*}

Now, since $\eta$ is generated by three antiholomorphic reflections, we are in presence of a Lagrangian triangular group~\cite{Will07,FK}: the three antiholomorphic reflections preserve a real plane inside $\mathbf H_\C^2$, also called \emph{Lagrangian} plane in complex hyperbolic geometry.

\section{Appendix: figures and code}

\begin{figure}[ht]
\centering
\includegraphics[width=\textwidth]{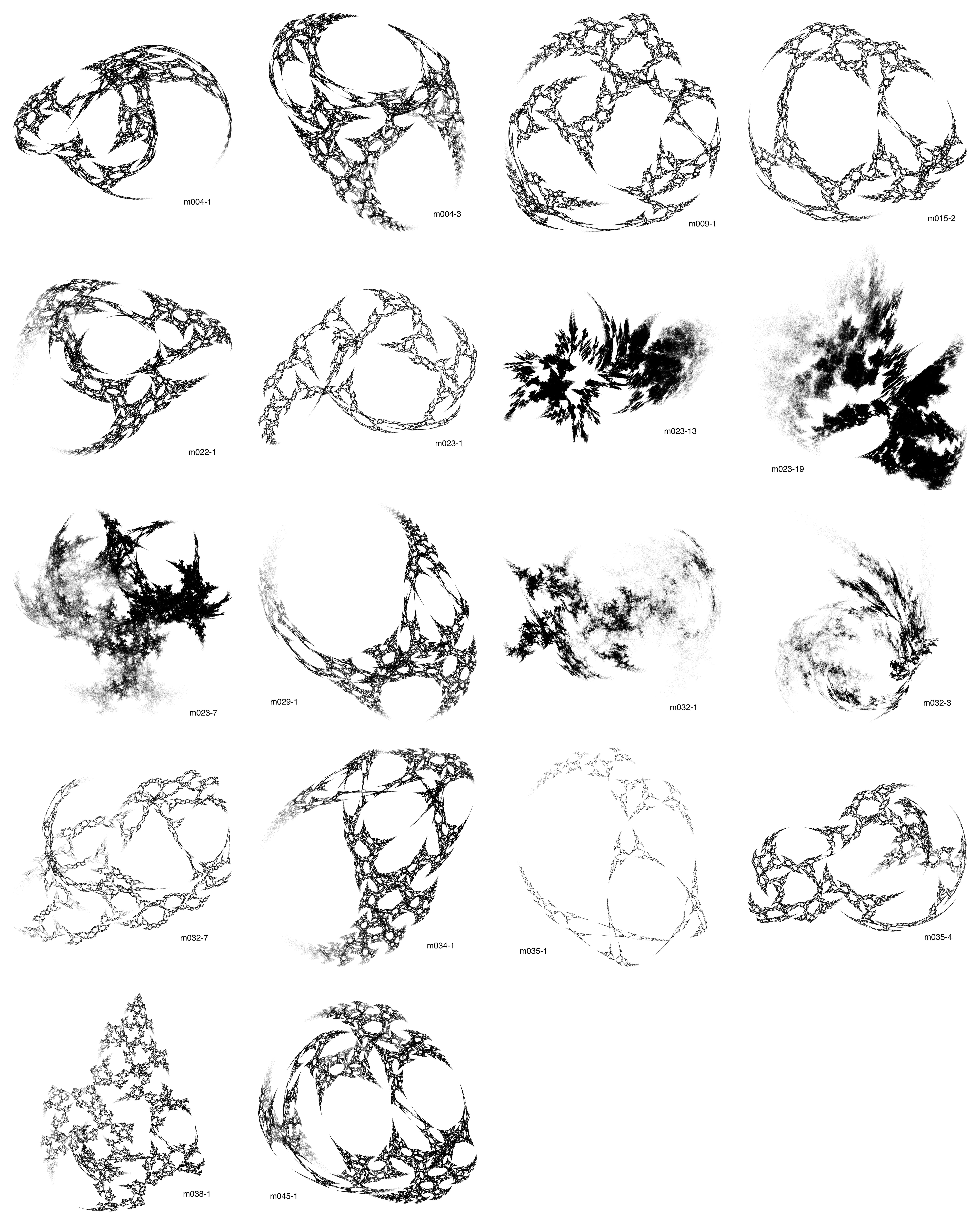}
\caption{The fractal limit sets found in the census, part 1.}\label{fig-all-1}
\end{figure}

\begin{figure}[ht]
\centering
\includegraphics[width=\textwidth]{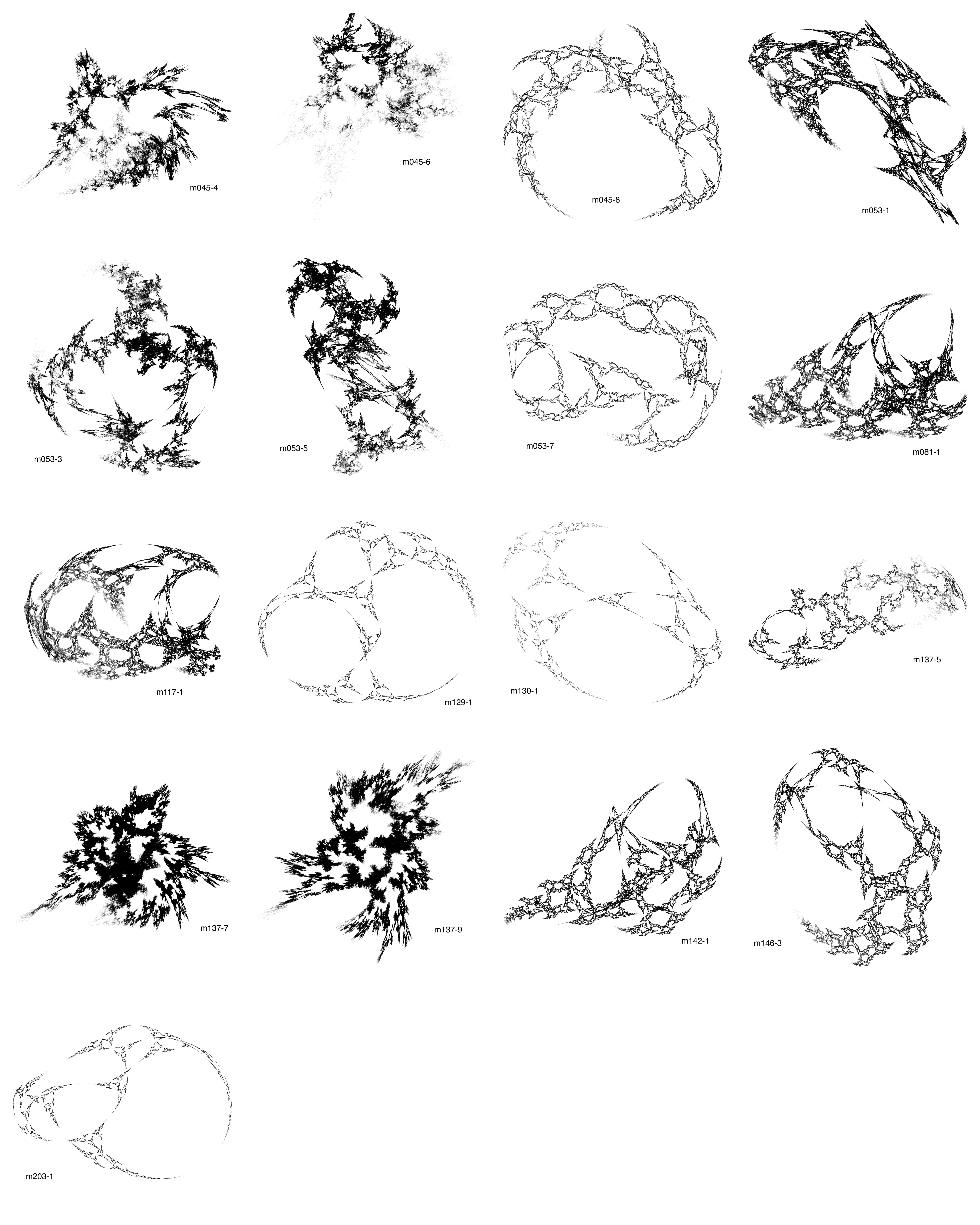}
\caption{The fractal limit sets found in the census, part 2.}\label{fig-all-2}
\end{figure}

\begin{figure}[ht]
\centering
\includegraphics[width=\textwidth]{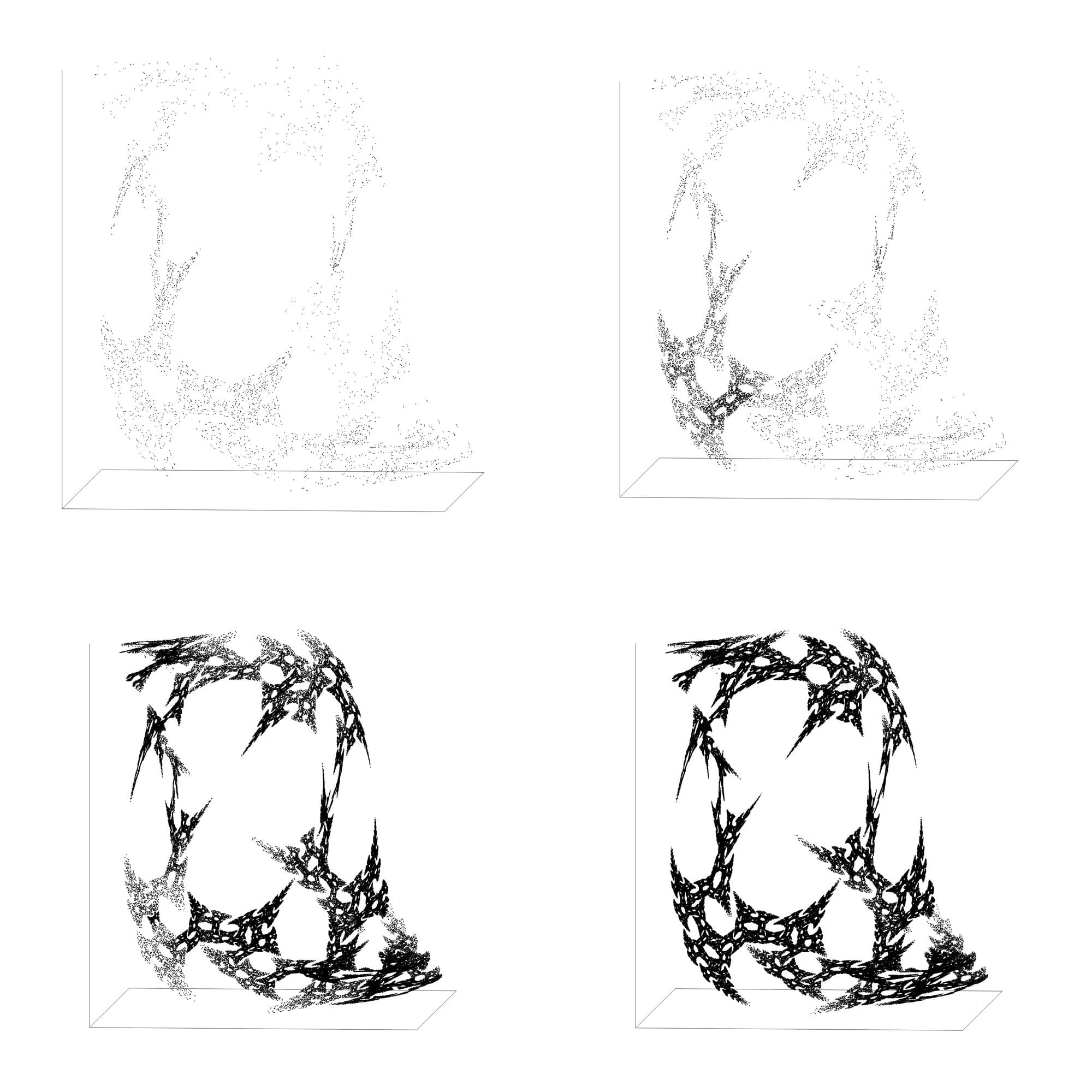}
\caption{Limit set of m004-1, with steps illustrated.}\label{fig-m004-1}
\end{figure}

\begin{figure}[ht]
\centering
\includegraphics[width=\textwidth]{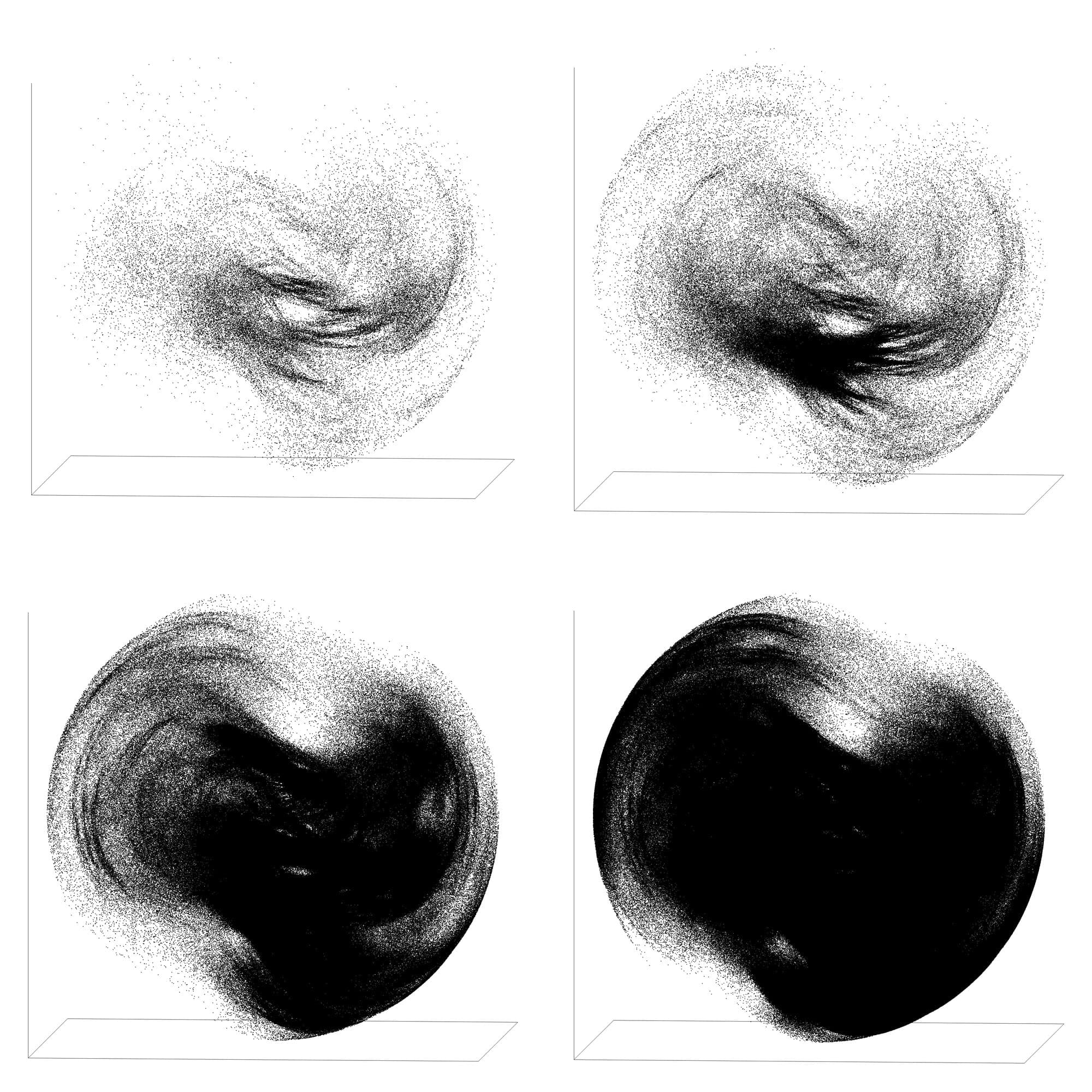}
\caption{Limit set of m004-5, with steps illustrated.}\label{fig-m004-5}
\end{figure}

\begin{figure}[ht]
\centering
\includegraphics[width=\textwidth]{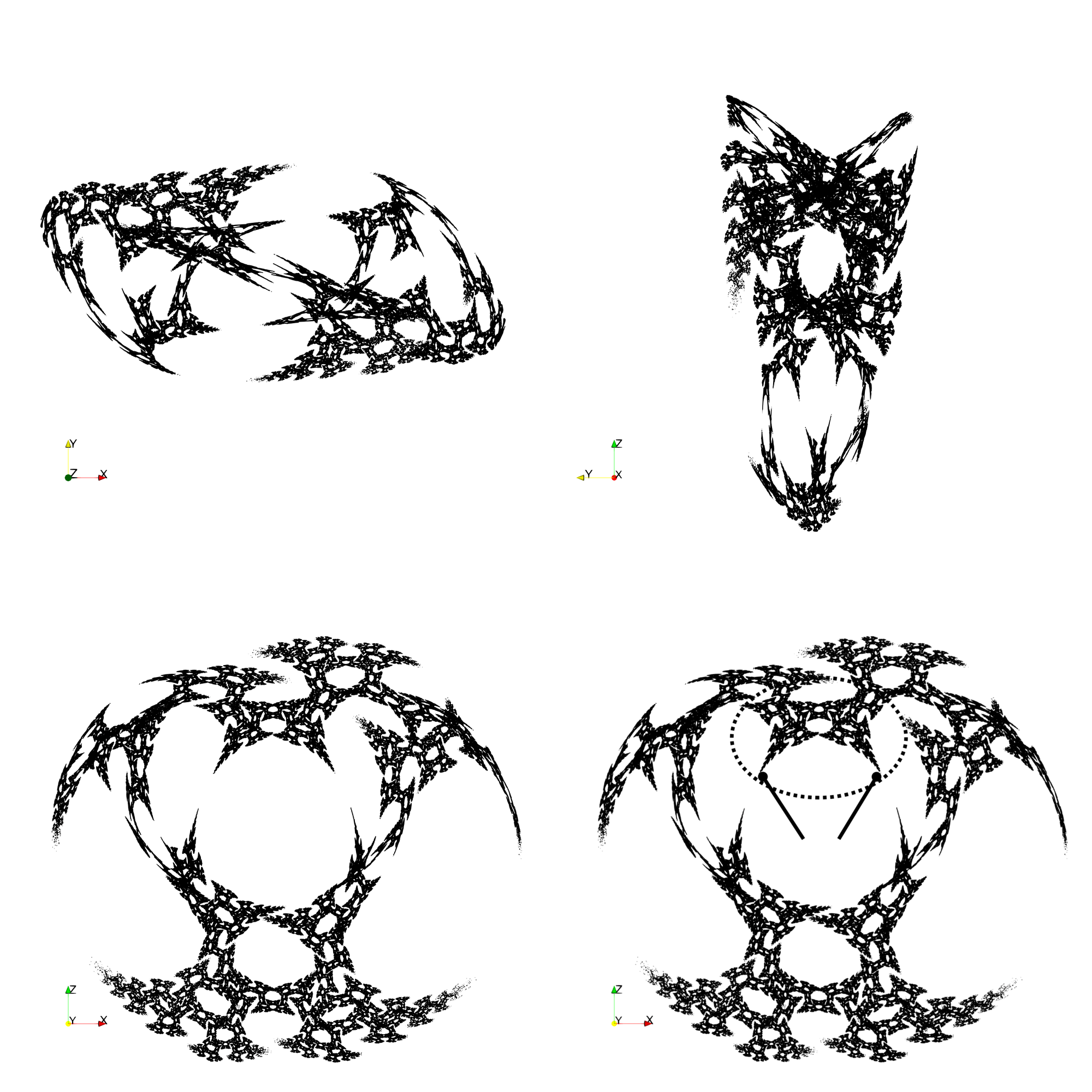}
\caption{Limit set of $\Delta(3,3,4;\theta_\infty)$.}\label{fig-2-1}
\end{figure}

\begin{figure}[ht]
\centering
\includegraphics[width=\textwidth]{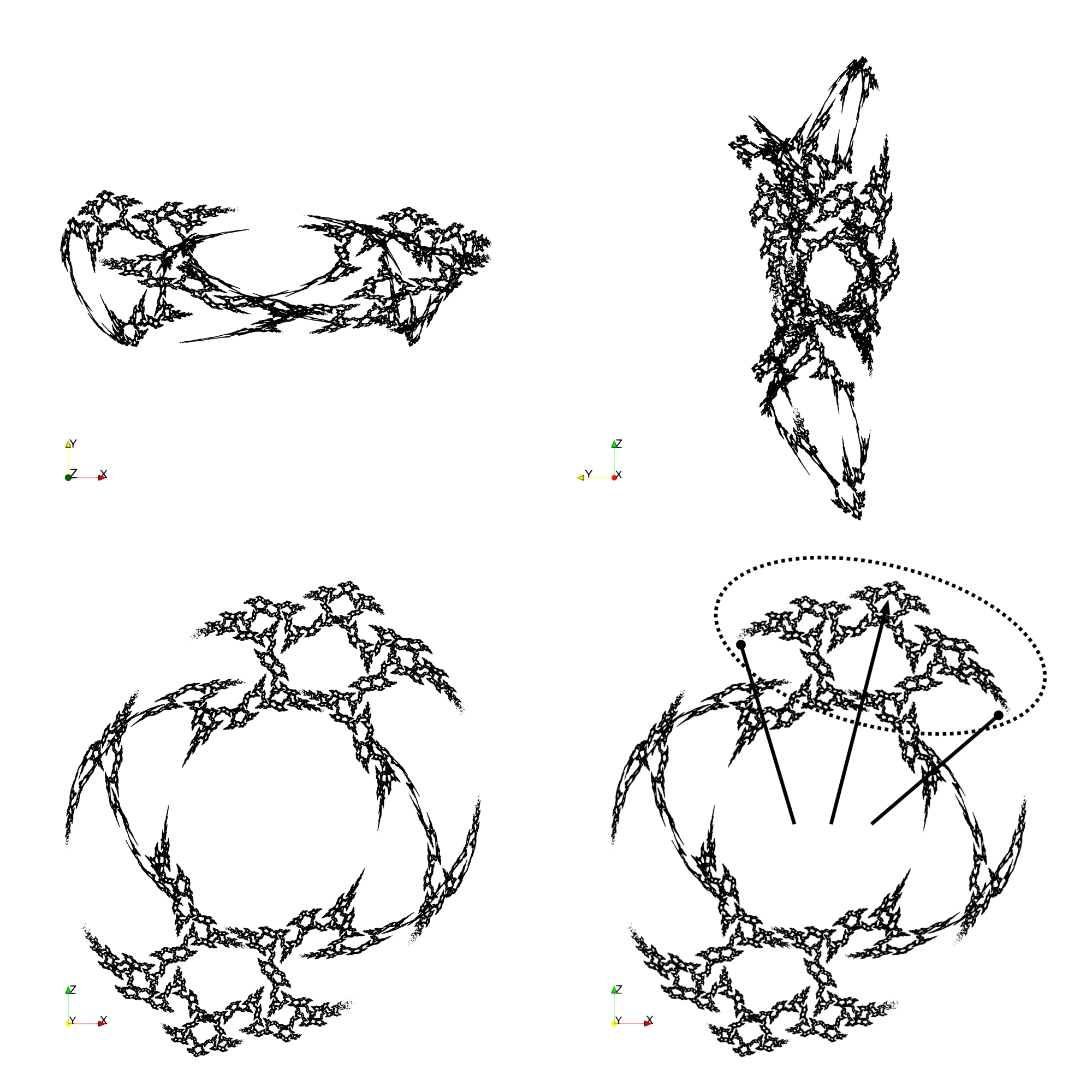}
\caption{Limit set of $\Delta(3,3,5;\theta_\infty)$.}\label{fig-2-2}
\end{figure}

\begin{figure}[ht]
\centering
\includegraphics[width=\textwidth]{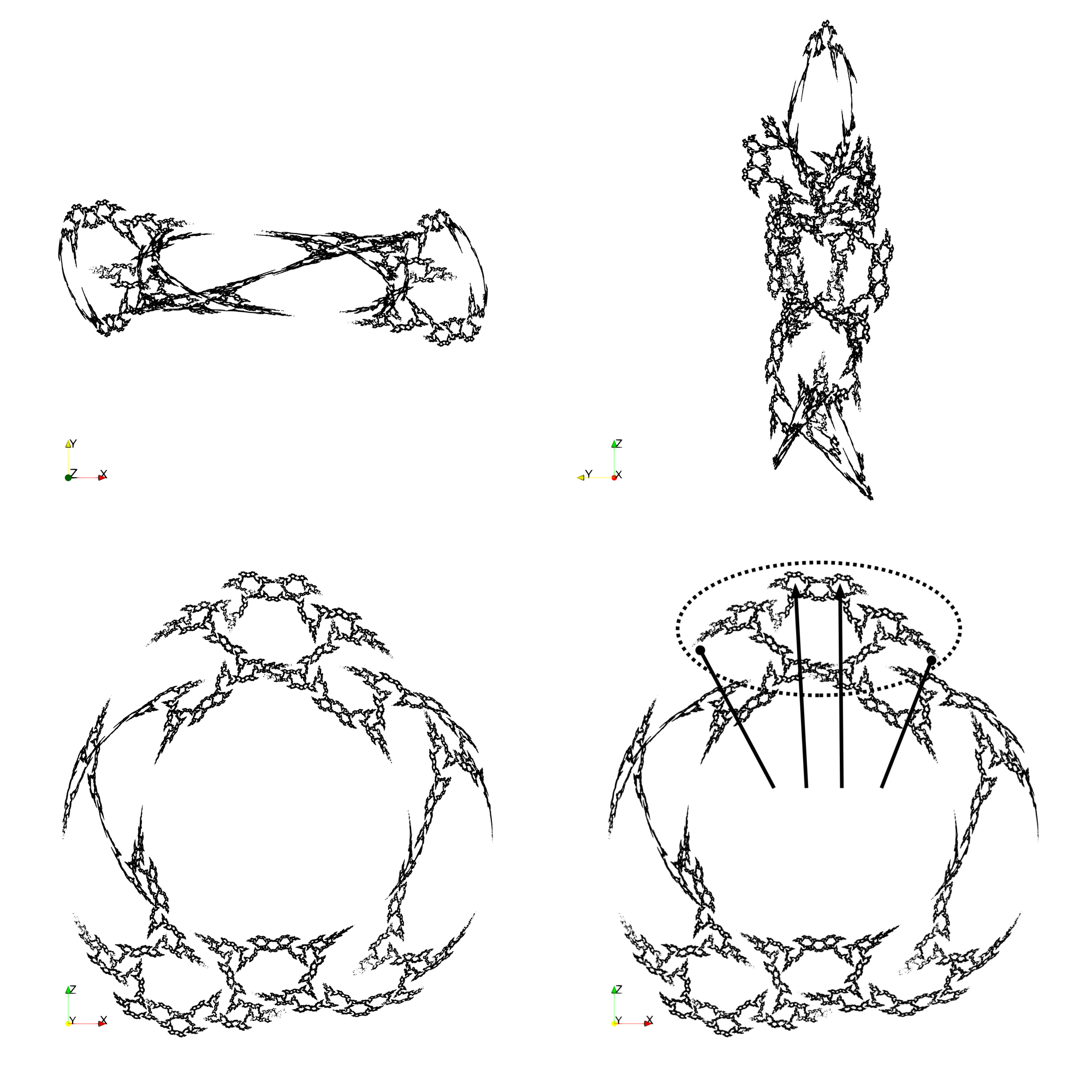}
\caption{Limit set of $\Delta(3,3,6;\theta_\infty)$.}\label{fig-2-3}
\end{figure}

\begin{figure}[ht]
\centering
\includegraphics[width=\textwidth]{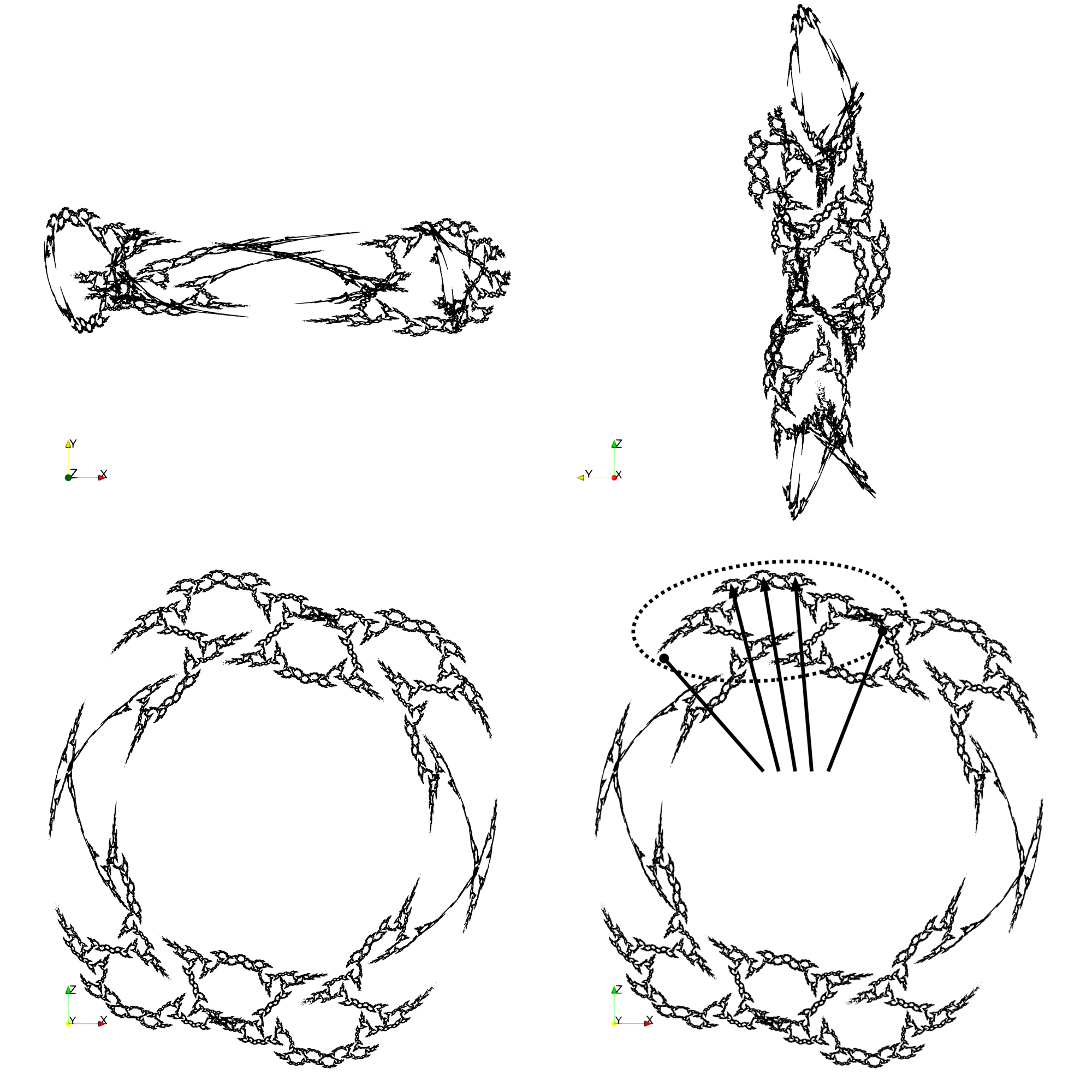}
\caption{Limit set of $\Delta(3,3,7;\theta_\infty)$.}\label{fig-2-4}
\end{figure}

\begin{figure}[ht]
\centering
\includegraphics[width=\textwidth]{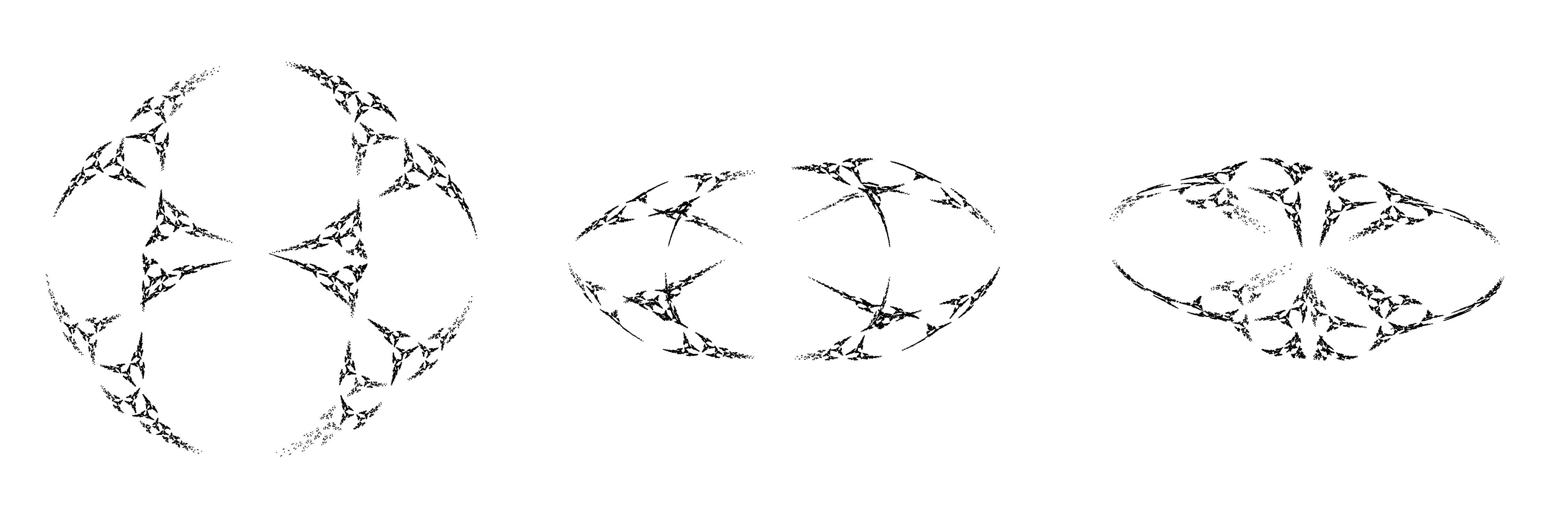}
\caption{Limit set of $\Delta(3,3,\infty;\theta_\infty)$.}\label{fig-3_3_inf}
\end{figure}

\begin{figure}[ht]
\centering
\includegraphics[width=\textwidth]{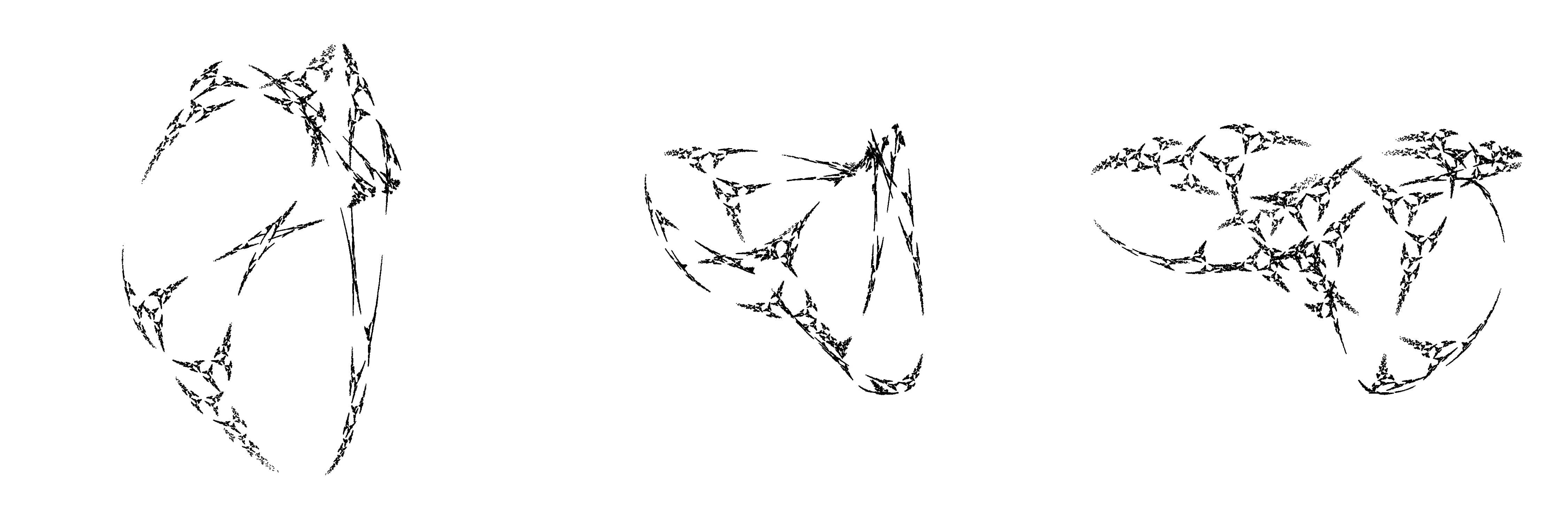}
\caption{Limit set of m035-1.}\label{fig-m035}
\end{figure}

\begin{figure}[ht]
\centering
\includegraphics[width=\textwidth]{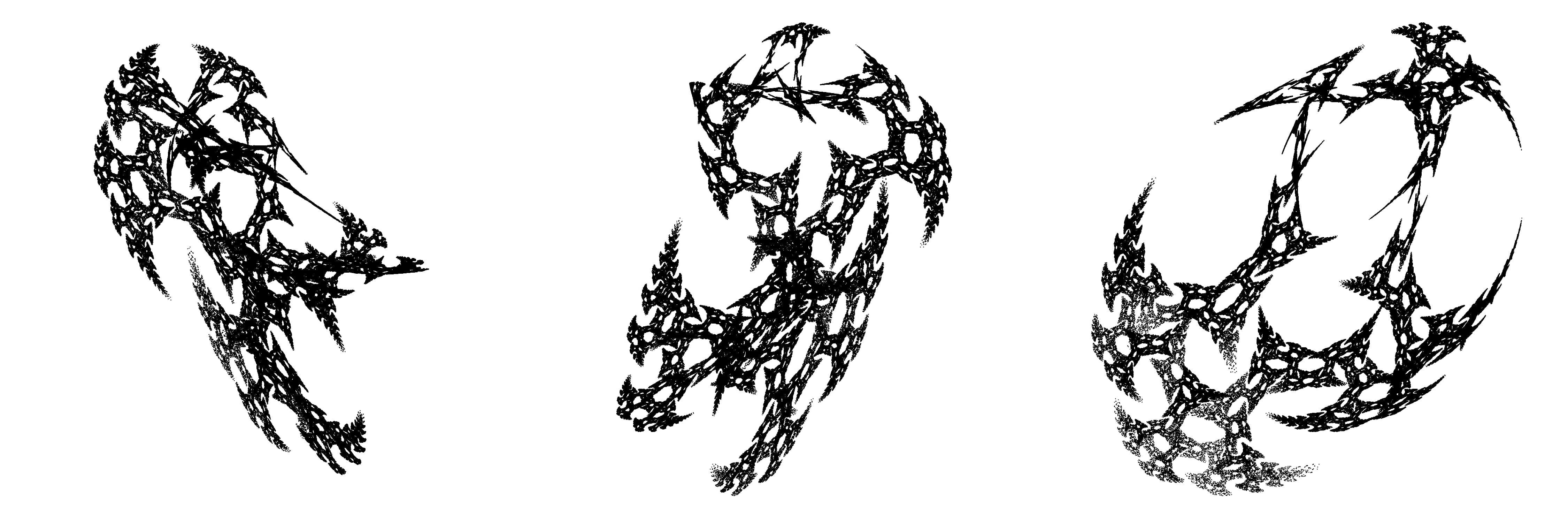}
\caption{Limit set of m045-1.}\label{fig-m045}
\end{figure}

\begin{figure}[ht]
\centering
\includegraphics[width=\textwidth]{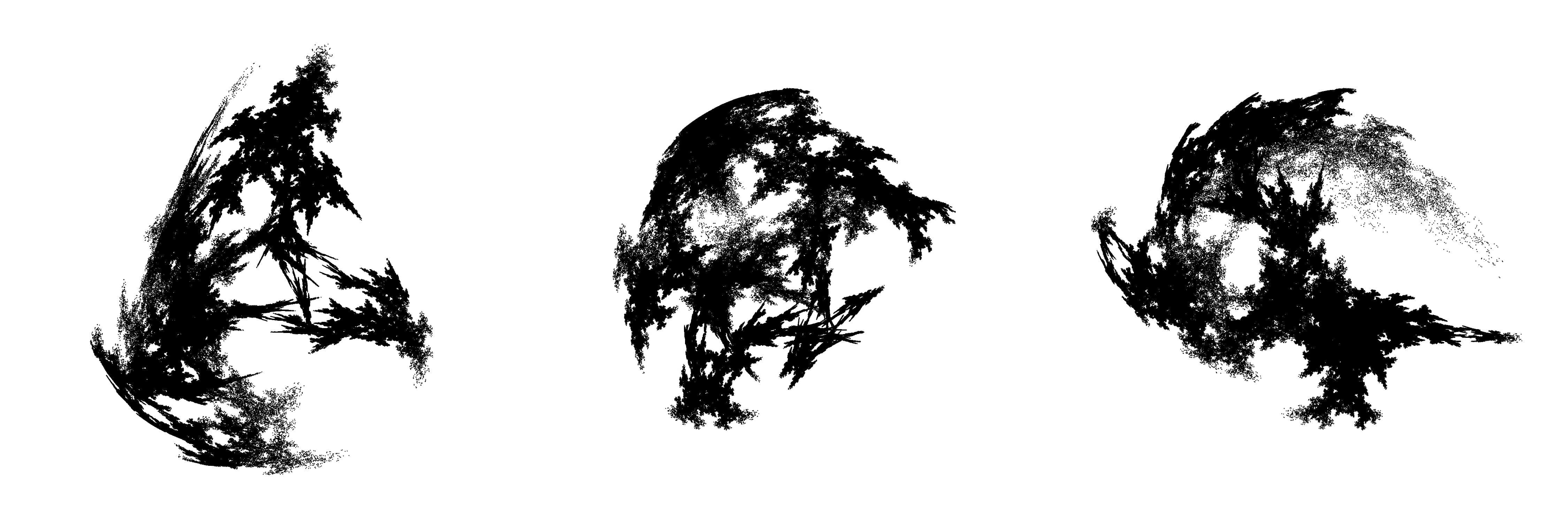}
\caption{Limit set of m023-7.}\label{fig-m023-7}
\end{figure}

\begin{figure}[ht]
\centering
\includegraphics[width=\textwidth]{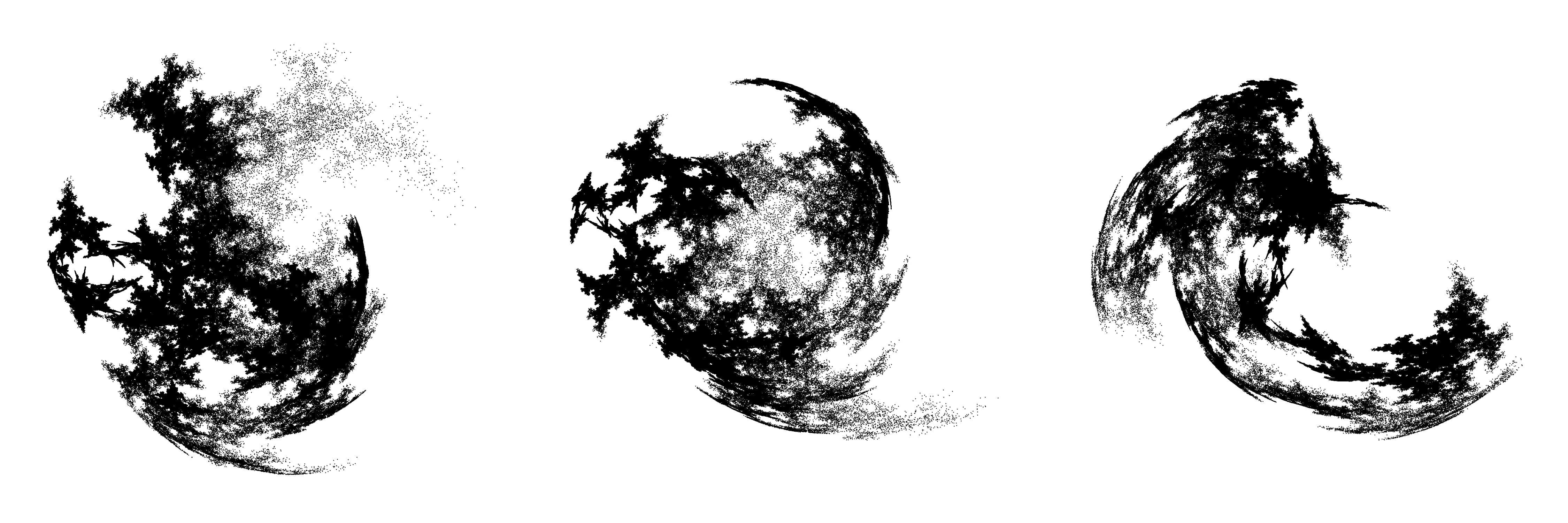}
\caption{Limit set of the triangle $\Lambda_2(3,3,4)\to\PU(2,1)$ with $[x,y]$ unipotent.}\label{fig-3_3_4-lag}
\end{figure}

\begin{python}
import SnapPy
import numpy

# Data
s = 'm009'
i,j = 0,1 # "id" in the table
#

M = SnapPy.Manifold(s)
G = M.fundamental_group()
print(G,G.peripheral_curves())

P = M.ptolemy_variety(3,'all').retrieve_solutions(prefer_rur=True)
S = [[component
      for component in per_obstruction
      if component.dimension == 0]
     for per_obstruction in P]
K = S[i][j]

def f(x): # evaluates words
    mat_x = K.evaluate_word(x,G)
    return [[z.lift() for z in y] for y in mat_x]
 
print(f('aaabaab'))
\end{python}

\printbibliography

\end{document}